\documentclass[11pt,leqno,A4paper]{amsart}
\usepackage{amsmath}
\usepackage{amsfonts}
\usepackage{amssymb}
\usepackage{graphicx}
\usepackage{hyperref}
\usepackage{color}
\usepackage{soul}

\newtheorem{theorem}{Theorem}

\newtheorem{conjecture}[theorem]{Conjecture}
\newtheorem{corollary}[theorem]{Corollary}

\newtheorem{definition}[theorem]{Definition}
\newtheorem{example}[theorem]{Example}

\newtheorem{lemma}[theorem]{Lemma}

\newtheorem{proposition}[theorem]{Proposition}
\newtheorem{remark}[theorem]{Remark}

\numberwithin{equation}{section}
\numberwithin{theorem}{section}

\newcommand{\supp}{\operatorname{supp}}

\usepackage{todonotes}
\newcommand{\Hm}[1]{\leavevmode{\marginpar{\tiny%
			$\hbox to 0mm{\hspace*{-0.5mm}$\leftarrow$\hss}%
			\vcenter{\vrule depth 0.1mm height 0.1mm width \the\marginparwidth}%
			\hbox to 0mm{\hss$\rightarrow$\hspace*{-0.5mm}}$\\\relax\raggedright#1}}}

\title[$L^{p}$-parabolicity]{${\bf L^{p}}$-parabolicity of Riemannian manifolds}

\author{Alexander Grigor'yan  \and Leandro F. Pessoa \and Alberto G. Setti}

{\address{Department of Mathematics\\University of Bielefeld\\
33501 Bielefeld, Germany.} \email{grigor@math.uni-bielefeld.de}}

{\address{Departamento de Matem\'{a}tica\\Universidade Federal do Piau\'{i}\\
64049-550, Teresina, Brazil.} \email{leandropessoa@ufpi.edu.br} }

{\address{Dipartimento di Scienza e Alta Tecnologia\\Universit\`a  dell'Insubria\\
via Valleggio 11, 22100 Como, Italy.} \email{alberto.setti@uninsubria.it} }

\begin{document}

\begin{abstract}
We describe the $L^p$-parabolicity of a Riemannian manifold in terms of a unified theory of nonlinear $L^p$-capacities, for the whole range $1\leq p\leq \infty$.   For example, the \(L^{1}\)-parabolicity is equivalent to the usual notion of parabolicity/recurrence, while \(L^{2}\)-parabolicity is equivalent to biparabolicity. For any $1\leq p\leq \infty,$ the $L^p$-parabolicity turns out to be equivalent to the $L^q$-Liouville property for positive superharmonic functions, where $p$ and $q$ are H\"older conjugate exponents.  We also provide a new capacitary characterization of the $L^1$-Liouville property. Finally we obtain an almost optimal volume growth condition for the $L^p$-parabolicity for $1<p\leq 2$ as well as  a sharp volume condition  for all $1<p<\infty$ for manifolds of non-negative Ricci curvature and  for  model manifolds.
\end{abstract}

\maketitle
\tableofcontents

\section{Introduction}

The properties of diffusion processes on Riemannian manifolds have been extensively investigated and have proven to be useful in the study of the geometric structure of the manifolds and the analytic properties of  differential operators such as the Laplace-Beltrami operator (see, e.g., \cite{alias_mastrolia_rigoli,grigoryan_bams,peter_li,pigola_rigoli_setti_memoirs,yau} and references therein).

One of such properties is recurrence: Brownian motion \(X_{t}\) on a manifold \(M\) is called \emph{recurrent,} and the manifold \(M\) is called \emph{parabolic},  if the process \(X_{t}\) starting at any point $x \in M$ visits infinitely often every neighborhood of $x$ almost surely. The notion of parabolicity plays a crucial role in the classification of Riemann surfaces (cf. \cite{ahlfors_sario}), and it reveals deep connections between Stochastic Analysis, Potential Theory, Analysis and Geometry.

Parabolicity admits several equivalent characterizations, such as the non-existence of a positive Green function of the Laplace-Beltrami operator, the validity of the Liouville property for positive superharmonic functions, and also the vanishing of the harmonic capacity of some/every pre-compact set, to quote a few (see, e.g., \cite{grigoryan_bams}), which makes this notion an indispensable tool  in Potential Theory and PDEs on Riemannian manifolds.

We say that a manifold \(M\) enjoys the \(L^{q}\)-\emph{Liouville property} for superharmonic functions if every positive superharmonic function in  \(L^{q}(M)\) is constant. For example, the \(L^{\infty}\)-Liouville property is equivalent to the parabolicity of \(M\) (since the minimum of a superharmonic function with a constant is again superharmonic).

At the other end of the scale,  the \(L^{1}\)-Liouville property for superharmonic functions is equivalent to the fact that the expected lifetime  of Brownian motion on $M$ is infinite.  This is closely related (although not equivalent) to the \emph{stochastic completeness} of the manifold, which means that the lifetime of Brownian motion is infinite with probability one (cf. \cite{grigoryan_bams} and \cite{bessa_pigola_setti,grigoryan_89,pessoa_pigola_setti}).

It is natural to extend the investigation of the $L^q$-Liouville property for positive superharmonic functions to the entire scale of exponents $1\leq q \leq \infty$, which is the main subject of this paper. In fact, our main  purpose 
is to introduce a unified theory of nonlinear capacities in the setting of Riemannian manifolds and to describe the $L^q$-Liouville property in terms of $L^p$-capacities, where  \(p\)  and \(q\) are H\"older conjugate exponents.

In \cite{adams_hedberg} D.R. Adams and L.I. Hedberg gave a systematic treatment of a general theory of $L^p$-capacities on 
Euclidean space as a generalization of the so-called $(\alpha,p)$-capacities related to nonlinear potentials developed in  \cite{fuglede_165,havin_mazya_202,havin_mazya_203,meyers_318,reshetnyak_362}. For the problems we want to deal with we will restrict our attention to the $L^p$-capacities defined in terms of the positive Green kernel of the manifold or of open subsets thereof, typically with Dirichlet boundary conditions.

Since on parabolic manifolds the \(L^{q}\)-Liouville theorem holds for all \(q\in[1,\infty]\), in what follows we assume  that \(M\) is non-parabolic, in particular, \(M\) possesses a positive Green function of the Laplace-Beltrami operator. Let \(1\leq p<\infty\). Denote by $g(x,y)$ the minimal positive Green function on $M$ and define the $L^p$-capacity of a compact subset $K\subset M$ by
\begin{equation}\label{def_cap_intro}
C_{p}(K)= \inf \left\lbrace \int_M f^p \colon f \in L^p_+(M), Gf(x) \geq 1 \ \ \text{for all} \ \ x \in K\right\rbrace
\end{equation}
where the potential $G f$ is given by
$$Gf(x) = \int_M g(x,y)f(y)d\mu(y).$$
In the limiting case $p=1$, the \(L^{1}\)-capacity is equivalent to the usual harmonic capacity $\text{Cap}(K)$ (see Remark~\ref{rmk_equiv_C1}). The equivalence of these two notions of capacity is valid also in the case when \(M\) is parabolic as in this case \(g\equiv\infty\) and, hence, the capacity of any compact set is zero.

We  say that a manifold $M$ is $L^p$-parabolic if 
every compact subset of \(  M\) has zero $C_{p}$-capacity. One of the main results of this paper, Theorem \ref{characterization_biparb_capacity}, says that, for any $1\leq \ p < \infty,$ the $L^p$-parabolicity is equivalent to the $L^q$-Liouville property. This result extends the aforementioned equivalence between 
parabolicity and the \(L^{\infty}\)-Liouville property.

Although the other limiting case, namely $p=\infty$, is not covered by the definition given in \eqref{def_cap_intro}, we obtain a characterization of the $L^1$-Liouville property in terms of a suitable \(L^\infty\)-capacity (see Theorem \ref{L1_liouville_cap}). This in turn allows us to define a notion of $L^\infty$-parabolicity equivalent to the $L^1$-Liouville property.

An important tool in establishing these equivalences, which is useful in its own right, is the characterization of the $L^q$-Liouville property, $1\leq q<\infty$, by means of the non-integrability of $g(x,\cdot)^q$ outside a pre-compact open set containing $x$, for some/every $x \in M$ (see Theorem \ref{equiv_G2}). As a consequence of this characterization we prove that the $L^2$-parabolicity coincides with the \emph{biparabolicity} of \(M\)
introduced in \cite{faraji_grigoryan} (see Subsection \ref{subsec_bip}). Besides, this allows us to show that the $L^p$-parabolicity implies the $L^{s}$-parabolicity for every $1\leq p \leq s < \infty$ (cf. Corollary \ref{cor_hier}).

As a further application of our results we also show that, on a complete Riemannian manifold $M$, if, for some $o \in M$ and sufficiently large $r>1,$ the volume \(V(r)\) of the ball $B_r(o)$ centered at $o$ of radius $r$ satisfies

\begin{equation}\label{sharp_vol_1}
 V(r)\leq C \frac{r^{2p}}{\log r},
\end{equation}
for some constant $C>0$, then $M$ is $L^p$-parabolic provided $1< p \leq 2$ (cf. Theorem \ref{vol_cond_general}).

This extends the nearly optimal sufficient volume condition established in \cite{faraji_grigoryan} for biparabolicity ($p=2$), and is compatible with the known volume conditions for the usual parabolicity in the limiting case $p = 1$. For manifolds with non-negative Ricci curvature we obtain a sufficient integral condition \eqref{Vp} in the entire range $1<p<\infty$ that implies the sharp volume condition
\begin{equation*}
 V(r)\leq C r^{2p}(\log r)^{p-1}.
\end{equation*}
Moreover, the integral condition \eqref{Vp} turns out to be always valid for general model manifolds (see Section \ref{sec_4}).


The paper is organized as follows. In Section \ref{sect_def_capacity} we describe the general theory of $L^p$-capacities on domains of a Riemannian manifold. In particular, we show that they satisfy natural monotonicity and limit properties with respect to exhaustions (cf. Proposition \ref{C_g2_test_functions}), and we establish some useful results valid in this setting, for instance, we provide a proof of the dual representation of the standard capacity $\text{\rm Cap}$ by means of non-negative Radon measures (see Theorem  \ref{parabolicity_measure_dual}), and we establish, in the Riemannian setting, that the Laplacian capacity $\text{\rm Cap}_{2,1}$ is twice $\text{\rm Cap}$  (see Theorem \ref{BMP-Thm}), thereby extending a theorem of H. Brezis, M. Marcus, and A. C. Ponce, \cite{brezis_marcus_ponce}. Finally, in Theorem \ref{Cap2,2_C_g2} we prove the equivalence between the $L^p$-capacity and the Laplacian capacity $\text{\rm Cap}_{2,p}$, for $1\le p< \infty$. These results are obtained using  two types of approximation: one based on heat semigroup regularization, and the \(L^1\)-approximation based on integration against an integral kernel built from the Green function introduced in \cite{pigola_valtorta_veronelli_NLA}.


Our main results are stated in Section \ref{sec_3}. We establish there the connection between the $L^p$-parabolicity and the non-integrability of $g(x,\cdot)^q$, which yields the equivalence between the $L^p$-parabolicity and the $L^q$-Liouville property. Besides, it also shows that a pre-compact open set has zero $L^{p}$-capacity if and only if this holds for every such set (cf. Corollary \ref{cor_inv_zero_cap}). We end this section by showing the equivalence between biparabolicity and $L^2$-parabolicity.

Section \ref{sec_L1} is devoted to the limiting case $q=1$. There, we introduce a $C_{\infty}$-capacity and, by describing the extremal capacitary measure of compact sets in terms of the mean exit time of the manifold, we obtain a capacitary characterization of the $L^1$-Liouville property (see Theorem \ref{L1_liouville_cap}). We end the section collecting several properties of  the $C_{\infty}$-capacity which in particular imply that it is a Choquet capacity (cf. Proposition \ref{prop_prop_C_g_infty}). The volume growth conditions for $L^p$-parabolicity are discussed in Section \ref{sec_4}. Finally, in Appendix \ref{appendix}, we describe the regularization procedures used in the proofs of Theorems \ref{parabolicity_measure_dual} and \ref{Cap2,2_C_g2}, and present the proof of Theorem \ref{BMP-Thm}.

\section{$L^{p}$-capacity}\label{sect_def_capacity}

Let $(M,\mu)$ be a Riemannian manifold without boundary endowed with the usual Riemannian measure $\mu$ and, for $1\leq p < \infty$, denote by $L^p_+(M)$ the set of non-negative functions in $L^p(M)$. The general theory of nonlinear capacities presented in \cite[Sect.2.3]{adams_hedberg} is based on considering  generic kernels, that is, non-negative functions $g \colon M\times M \rightarrow \mathbb{R}$, such that  $g(x,\cdot)$ is lower semicontinuous on $M$ and $g(\cdot,y)$ is measurable on $M$.

In what follows we will restrict our attention to subsets $U\subset M$ that admit a finite positive Dirichlet Green kernel $g^U(x,y)$ with zero boundary data, which is obtained by the usual exhaustion procedure (see e.g. \cite{dodziuk,pessoa_pigola_setti}), and satisfies $g^U(x,y) \leq g^V(x,y)$ for every $x,y \in U\subset V$. In particular, when $U=M$, we denote by $g(x,y)$ the positive minimal Green kernel of $M$.

Let $\mathcal{M}^+(U)$ be the set of non-negative Radon measures supported in $U$. For $\nu \in \mathcal{M}^+(U)$ define the potential $G^U\nu$ by
\begin{eqnarray*}
G^U\nu(y) &=& \int_U g^U(x,y)d\nu(x), \quad y \in U.
\end{eqnarray*}
If $\nu = f\mu$ for some measurable non-negative function $f$, we will write $G^Uf = G^U (f\mu)$. The potential $G^U\nu$ is well defined provided we allow it to take the value $\infty$. 


Based on \cite[Sect.2.3]{adams_hedberg}, we define the relative variational $L^p$-capacity of the capacitor $(E,U)$ as follows.


\begin{definition}\label{def_bicapacity}
For every $1\leq p<\infty$ and $E \subset U$, let
$$ \Omega^U_E = \{f \in L^p_+(U) \colon G^Uf(x) \geq 1 \ \ \text{for all} \ \ x \in E\},$$
and define
\begin{equation}\label{def_cap_var}
C_{p}(E,U) = \inf_{\Omega^U_E} \int_{U} f^pd\mu .
\end{equation}
We set $C_{p}(E,U) = \infty$ whenever $\Omega^U_E = \emptyset$, and write $C_{p}(E)$ instead of $C_{p}(E,M)$ if $U = M$.
\end{definition}


This set function is an outer capacity, countably  subadditive (see \cite[Sec.2.3]{adams_hedberg}), and it is easily seen to be monotone in $E$, that is, if $E_1 \subset E_2 \subset U$ then
$$ C_{p}(E_1,U) \leq C_{p}(E_2,U) .$$
Furthermore, $C_{p}(E,U)$ is a Choquet capacity for every $1<p<\infty$ (cf. \cite[Prop.2.3.12]{adams_hedberg}).
The usual monotonicity property with respect to $U$  follows from the fact that if $f\in L^p_+(U_1)$ is such
that $G^{U_1}f\geq 1$ on $E$,
extending $f$ by zero on
$U_2\backslash U_1$, by
monotonicity, $G^{U_2}f\geq 1$ and
therefore
$\Omega^{U_1}_E \subseteq
\Omega^{U_2}_E$.

We recall the following characterization of subsets with null capacity.

\begin{proposition}[Prop. 2.3.7 in \cite{adams_hedberg}]\label{characterizing_capacity_null}
Let $E \subset U$ and $1\leq p < \infty$. Then $C_{p}(E,U) = 0$ if and only if there is $f \in L^p_+(U)$ such that $$E \subset \{x \in U \colon G^Uf(x) \equiv \infty\}.$$
\end{proposition}

We shall say that a property holds $C_{p}$-quasieverywhere, $C_{p}$-q.e. for short, if it holds except on a subset $E$ with $C_{p}(E,U) = 0$.

\begin{remark}\label{general_sign}
{\rm One can extend the definition of $G^Uf$ to arbitrary measurable functions $f$ by setting
$$ G^U f(x) = G^U f^+(x) - G^U f^-(x),$$
whenever at least one of the summands on the right hand side is finite. Then, $G^U f(x)$ is well defined and finite $C_{p}$-q.e. by Proposition \ref{characterizing_capacity_null}. It is easy to see that the restriction on the sign of the test functions in $\Omega^U_E$ can be dropped, that is
$$C_{p}(E,U) = \inf\left\lbrace \vert\vert f\vert\vert^p_{L^p(U)} \colon f \in L^p(U), G^U f \geq 1  \ \ C_{p}\text{-q.e} \ \ \text{on} \ \ E \right\rbrace.$$}
\end{remark}

The main advantage in using $L^p$ spaces, for $1< p<\infty$, is to have the existence of a capacitary function. Indeed,  since $\Omega^U_E$ is a convex subset of the uniformly convex Banach space $L^p(M)$, by standard Functional Analysis results there exists an extremal function $f^E$ belonging to the closure of ${\Omega}^U_E$ in $L^p(M)$ that minimizes the above capacity, i.e.,
$$ C_{p}(E,U) = \int_U (f^E)^pd\mu .$$
The function $f^E$ is called the capacitary function of $E$, and $G^Uf^E$ is called the capacitary potential of $E$. Furthermore, $\overline{\Omega^U_E}$ can be explicitly described as (cf. \cite[Prop.2.3.9]{adams_hedberg})
\begin{eqnarray*}
\overline{\Omega^U_E} = \{f \in L^p_+(U) \colon G^Uf(x) \geq 1 \ \ C_{p}\text{-q.e. on } \ E\}.
\end{eqnarray*}

In the setting of Riemannian manifolds we have the following density result which will be instrumental in establishing connections between different nonlinear capacities.

\begin{proposition}\label{C_g2_test_functions}
Let $U\subset M$ be an open set and let $E\subset U$ be a compact subset. Then, for every $1\leq p<\infty$, it holds
\begin{equation}\label{c_g_p_smooth_test}
C_{p}(E,U)=\inf\left\lbrace\int_U\phi^p \,d\mu\,:\, 0\leq \phi\in C^{\infty}_c(U) ,\,\, G^U\phi \geq 1 \text{ on } E \right\rbrace.
\end{equation}
\end{proposition}
\begin{proof}
By definition, for every $\epsilon>0$, there exists  $f\in L^p_+(U)$
such that
\[
G^Uf(x)\geq 1 \text{ on }E\, \text{ and }\,
\int_U f^pd\mu < C_{p}(E,U)+\epsilon/2.
\]
Since
\[
\int_U f^p  d\mu =\inf_{c>1}\int_U (cf)^p  d\mu
\]
and, by homogeneity, $G^U(cf)(x)=cG^Uf(x)$, choosing  $c>1$ close enough to $1$  and replacing $f$ with $cf$ show that for  every $\epsilon>0$ there exists $f\in L^p_+(U)$ such that
\begin{equation}
\label{f-test}
\int_U f^pd\mu < C_{p}(E,U)+\epsilon \qquad \text{and} \qquad G^Uf(x)\geq c>1  \text{ on } E.
\end{equation}
Finally, let  $U_n$ be an exhaustion of $U$ by an increasing  sequence of pre-compact open sets and set
\[
f_n= \min\{n,f\}\mathbf{1}_{U_n},
\]
where $f$ is a function as in \eqref{f-test}. Then for every $n$, $f_n\in L^p_+(U)\cap L^\infty(U)$ with compact support in $U$.
Since $f_n\nearrow f$, by monotone convergence,
\[
\int_U f_n^pd\mu \nearrow \int_U f^p d\mu \]
and
\[G^Uf_n(x)\nearrow G^Uf(x) \,\,\text{ pointwise on } U.
\]
We claim  there exists $n_o$ such that $G^Uf_n(x)>(1+c)/2$ for every $x\in E$ and $n\geq n_o$.
Indeed, for every $n$ let
\[
V_n=\{ x\in U\,:\, G^Uf_n(x)>(1+c)/2\}.
\]
Then, $V_n$ is open, since $G^U f_n$ is lower semi-continuous, $V_n\subseteq V_{n+1}$ and $E\subset \cup_n V_n$, since $G^Uf_n(x)\to G^Uf(x)\geq c>(1+c)/2$ for every $x\in E$. By compactness there exists $n_0$ such that $E\subset V_{n}$ for every $n\geq n_0.$

By taking $n$ large enough, and replacing $f$ with $f_n$ we conclude that for every $\epsilon>0$ there exists $f\in L^p_+(U)\cap L^\infty_c(U) $ satisfying 
\begin{equation}
\label{f-test1}
\int_U f^pd\mu<C_{p}(E,U)+\epsilon\,\, \text{ and }\,\, G^Uf\geq1+\delta>1 \text{ on } E,
\end{equation}
for some $\delta>0$.

Suppose now that $U$ is bounded with smooth boundary. Fix $1<\tau<m/(m-2)$ and let $m/2<\tau' <\infty$ be its conjugate exponent. Let $f$ be as in \eqref{f-test1} and let $0\leq \phi_k\in C^\infty_c(U)$ be such that
\[
\phi_k\to f \quad \text{ in }\,\, L^{\tau'}(U),
\]
and therefore in every $L^s(U)$ with $s\leq \tau'$. Thus
\[
\int_U \phi_k^p d\mu \to \int_U f^p d\mu<C_{p}(E,U)+\epsilon.
\]

On the other hand, we claim that $h \to G^Uh$ defines a bounded operator from $L^{\tau'}(U)$ into $L^\infty(E)$.
Assuming the claim, we conclude that
\[
||G^Uf - G^U \phi_k||_{L^{\infty}(E)}\leq C||f-\phi_k||_{L^{\tau'}}\to 0, \, \text{ as }\, k \to \infty,
\]
so that, for $k$ sufficiently large
\[
G^U\phi_k(x)>(1+c)/2>1 \quad \forall x\in E,
\]
and \eqref{c_g_p_smooth_test} holds for $U$ bounded with smooth boundary.

It remains to prove the claim. Indeed, according to \cite[Thm. 4.17]{aubin}, there exists a constant $C$ which depends only on the distance of $x$ from $\partial U$ such that
\begin{equation}
\label{g^U estimate} g^U(x,y)\leq C d(x,y)^{2-m}\quad \forall y\in U,\footnote{When $m=2$, we use the logarithmic estimate for the Green function.}
\end{equation}
and it follows that for every $x\in E$ and fixed $1<s<m/(m-2)$
\[
\begin{split}
\int_U g^U(x,y)^s d \mu(y)&\leq C^s\int_{B(x,R)} d(x,y)^{s(2-m)}d\mu(y)\\
&\leq C^s \int_0^R r^{s(2-m)}\left(B^{-1}\sinh(Br)\right)^{m-1}dr=C_1(s)^s,
\end{split}
\]
where $R=\operatorname{diam} U$ and $B>0$ is such that
\[
\operatorname{Ric}_M\geq -(m-1)B^2 \quad \text{on } \ \{x\,:\, d(x,U)< R+1\}.
\]
Thus, if $s'$ is the exponent conjugate to $s$ and  $h\in L^{s'}$,  for every $x\in E$,
\[
\begin{split}
|G^Uh(x)|
&=
|\int_U g^U(x,y) h(y)d\mu(y)|\\
&\leq
\left(
\int_U g^U(x,y)^sd\mu(y)
\right)^{1/s}
||h||_{L^{s'}(U)}=C_1(s) ||h||_{L^{s'}(U)},
\end{split}
\]
and the claim follows.

To conclude, let $U$ be a general  open set which admits a positive Dirichlet Green kernel $g^U=g$, and  let $U_n$ be an increasing exhaustion of $U$ by pre-compact open sets with smooth boundary with $U_1\supset E$. Recall that, denoting by $g_n$ the Dirichlet Green kernel of $U_n$, we have (essentially by definition)
\[
g(x,y)=\lim_n g_n(x,y).
\]
Let $f$ be a function as in \eqref{f-test1} and let $n_1$ be such that $\operatorname{supp} f\subset U_{n_1}$. Since
$g_{n}(x,y)\nearrow  g(x,y)$ in $U_{n_1}$, for every $n\geq n_1$ and $x\in E$,
\[
 G^{U_{n}}f(x) = \int_{U_{n_1}} g_{n}(x,y) f(y) d\mu \nearrow
 \int_{U_{n_1}} g(x,y) f(y) d\mu
=G^{U}f(x),
\]
and arguing as above we deduce that for sufficiently large  $n$, $G^{U_n}f(x)>1+\delta/2>1$ on $E$ and therefore
\[
C_{p}(E,U) \leq C_{p}(E,U_n) \leq \int_{U_{n_1}} f^p d\mu\leq C_{p}(E,U)+\epsilon.
\]
Since we have proved that
\[
C_{p}(E,U_n) =\inf\left\lbrace\int_{U_n}\phi^p d\mu \,:\,  0\leq \phi\in C^\infty_c(U_n)  \text{ and } G^{U_n}\phi(x)\geq 1\right\rbrace
\]
the required conclusion follows.
\end{proof}

We note that in the last part of the above proof we have also shown the following.
\begin{proposition}
\label{prop_limit_regularity}
Let $U$ be an open set of $M$ and let $K\subset U$ be a compact set. Let also $U_n$ be an increasing exhaustion of $U$ by pre-compact open sets with $K\subset  U_1$ and denote by $g(x,y)$, $g_n(x,y)$ the Dirichlet Green kernels of $U$, $U_n$, respectively. Then
\[
C_{p}(K,U)=\lim_{n\rightarrow \infty} C_{p}(K,U_n).
\]
\end{proposition}

\begin{remark}
{\rm For $1<p<\infty$, $C_{p}(\cdot, U)$ is a Choquet capacity, thus Borel subsets are capacitable, and the conclusions of Propositions \ref{C_g2_test_functions} and \ref{prop_limit_regularity} hold for every Borel subset $E\subset U$.}
\end{remark}

The capacity we have introduced admits a dual representation by means of non-negative Radon measures. This duality is possible due to a general version of the von Neumann minimax theorem (cf. \cite[Thm.2.4.1]{adams_hedberg}). For any $1<p<\infty$, let $1<q<\infty$ be its H\"older conjugate exponent, that is, $pq = p + q$.

\begin{theorem}[Thm.2.5.1 in \cite{adams_hedberg}]\label{dual_cap_thm}
Let $K \subset U$ be a compact subset, and let $1<p<\infty$. Then
\begin{equation}\label{dual_cap_def}
C_{p}(K,U)^{\frac{1}{p}} = \sup \{ \nu(K) \colon \nu \in \mathcal{M}^+(K), \vert\vert G^U \nu \vert\vert_{L^{q}(U)} \leq 1\}.
\end{equation}
\end{theorem}

The set function defined by the right hand side of \eqref{dual_cap_def}, replacing compact subsets $K$ by arbitrary Borel subsets $E$, is easily seen to be monotone in $E$, and as a consequence of Theorem \ref{dual_cap_thm}, it satisfies the conditions of the capacitability given in \cite[Prop.2.3.12]{adams_hedberg}. It follows that the equality in \eqref{dual_cap_def} holds for every Borel subset (see also \cite[Cor.2.5.2]{adams_hedberg}).

For $1<p<\infty$, this dual representation also admits an extremal measure $\nu^K \in \mathcal{M}^+(K)$ realizing the $C_{p}$-capacity of a compact set $K \subset U$, whose potential is the extremal function $f^K \in L^p_+(U)$.

\begin{corollary}[Thm.2.5.3 in \cite{adams_hedberg}]\label{cor_extremal_measure}
Let $K\subset U$ be a compact set, and $1<p<\infty$. Then there is $\nu^K \in \mathcal{M}^+(K)$, called the capacitary measure of $K$, and a capacitary function $f^K = (G^U\nu^K)^{q-1}$, such that $G^U f^K(x) \geq 1 \ C_{p}$-q.e on $K$, and
$$ \nu^K(K) = \int_U \left(G^U \nu^K\right)^{q} d\mu = \int_U G^U f^K d\nu^K = C_{p}(K,U).$$
\end{corollary}

\rm We note that, in the limit case $p=1$, equality \eqref{dual_cap_def} does not hold. In this case, the right hand side of \eqref{dual_cap_def} corresponds to the standard capacity
$$\text{\rm Cap}(K,U)  =
\inf\left\{ \int_U |\nabla u |^2 d\mu \,:\, u\in W^{1,2}_0(U),\,\, u\geq 1 \text{ on } K\right\} .$$ Indeed we have
\begin{theorem}\label{parabolicity_measure_dual}
Let $U$ be a bounded open set and let $K\subset U$ be a compact domain. Then
 \begin{equation}\label{eq_Cap_C}
\begin{split} \text{\rm Cap}(K,U) 
&= \sup \{ \nu(K) \colon \nu \in \mathcal{M}^+(K), \vert\vert G^U \nu \vert\vert_{L^{\infty}(U)} \leq 1\}.
\end{split}
\end{equation}
\end{theorem}
\begin{proof}
By the regularity properties of $\text{Cap}(K,U)$, we may assume that $K$ and $U$ are bounded open sets with smooth boundary. Denoting by $C(K,U)$ the right hand side of the above equality, we begin by showing that  $C(K,U)\leq \text{Cap}(K,U)$. To this end, let  $\nu \in \mathcal{M}^+(K)$ be a Radon measure supported on $K$ and assume that  $h= G^U \nu$ satisfies  $\vert\vert  h 
 \vert\vert_{L^{\infty}(U)} \leq 1$. 

Fix $\epsilon >0$. Since  
the infimum in the definition of $\text{Cap}(K,U)$ can be taken over all functions $\varphi\in C^\infty _c(U)$ with $0\leq \varphi \leq 1$ and   $\varphi=1$ in a neighborhood of $K$, there  exists a function $\varphi_\epsilon$ 
such that 
\[
\int_U |\nabla \varphi
_\epsilon|^2 d\mu <\text{Cap}(K,U) + \epsilon.
\]
As $\nu$ is supported on $K$ and   $\varphi_\epsilon=1$  in a neighborhood of $K$  we have 
\[
\nu(K)= \int \varphi_\epsilon d\nu= \langle \varphi_\epsilon, - \Delta h\rangle= - \int_U \Delta \varphi_\epsilon h \,d\mu= \int_U \langle \nabla \varphi_\epsilon,\nabla h\rangle ,
\]
where we have used the fact that, since $\triangle h = - \nu$, $h$ is harmonic and therefore smooth in $U\setminus K$, and that $\varphi_\epsilon$ is constant in a neighborhood of $K$ and of $\partial U.$
According to Lemma~\ref{lemma:hW12}, $h\in W^{1,2}(U)$, so applying  the Cauchy-Schwarz  inequality we get
\begin{equation}
\label{eq:est-v(K)-1}
\nu(K)^2\leq \int_U|\nabla \varphi_\epsilon|^2 d\mu \int_U|\nabla h|^2 d\mu\leq (\text{Cap}(K,U) +\epsilon) \int_U|\nabla h|^2 d\mu.
\end{equation}
To conclude, we need to estimate the second factor on the right hand side. Let $h_t = P^U_th $ be the heat semigroup regularization induced by the Dirichlet heat kernel $p_t^U(x,y)$ and defined by 
\[
h_t(x)= \int_Up^U_t(x,y)h(y) d\mu.
\]
Since $h\in L^\infty(U),$ $h_t\in C^\infty(\overline{U})$ and vanishes on $\partial U$ (cf. Appendix \ref{appendix} for more details). Moreover, since $h\in W^{1,2}(U)$, by \cite[Ex.7.37 \& Thm.7.20]{grigoryan_book} $h_t\to h$ in $W^{1,2}$ as $t\to 0+$, and therefore
\begin{equation}
\label{eq:est-grad h-1} 
\int_U|\nabla h|^2 d\mu = \lim_{t\to 0+} \int_U|\nabla h_t|^2 d\mu.
\end{equation}
Let $\chi_n\in C^\infty_c(U)$ be a sequence of cut-off functions with $0\leq \chi_n\leq 1$ such  that $\chi_n\nearrow 1$ locally uniformly in $U$. Since $h_t$ vanishes on $\partial U$ by the divergence theorem  we have 
\begin{equation}
\label{eq:est-grad h-2} 
\begin{split}
    \int_U|\nabla h_t|^2 d\mu &= - \int_U h_t\Delta h_t d\mu\\
    &= -\lim_n \int_U \chi_nh_t \Delta (P_t h) d\mu
    = -\lim_n \int_U h \Delta (P_t(\chi_n h_t))  d\mu,
\end{split}
\end{equation}
where we have used that $P_t$ is symmetric and commutes with $\Delta$ on smooth compactly supported functions. We claim that 
\begin{equation}\label{eq:claim-hDelta f}
\int_U h \Delta f  d\mu = - \int_U f d\nu,  \,\, \forall 
f\in C^2(U)\cap C^1(\overline{U}) \,\,\text{ such that } \, f =0 \,\, \text{ on } \,\, \partial U.
\end{equation}
Notice that since $\Delta h = -\nu $ in the distributional sense, \eqref{eq:claim-hDelta f} holds, by definition of distributional Laplacian, whenever $f\in C_c^\infty(U)$ (in fact, $f\in C^2_c(U)$ suffices).

To establish \eqref{eq:claim-hDelta f} assuming only that  $f$ vanishes on $\partial U$, note that the boundary term in the Green's representation formula for $f$ vanishes so that 
\[
f(x)
=-\int_U g^U(x,y)    \Delta f(y)d\mu(y), 
\]
whence, by Fubini-Tonelli,
\[
\begin{split}
\int_Uh \Delta f  d\mu 
&= \int_U \Delta f (x) \int_U g^U(x,y)   d\nu(y)\\
&=  \int_U d\nu(y) \int_U \Delta f (x)  g^U(x,y)  d\mu(x)  
= -\int_U f(y)d\nu (y).
\end{split}
\]
Applying this with $f= P_t^U(\chi_n h_t)$ yields
\[
- \int_U h \Delta(P_t^U(\chi_n h_t)) d\mu= \int_U P_t^U(\chi_n h_t)d\nu 
\leq \nu(K) ||P_t^ U(\chi_n h)||_\infty \leq \nu(K) ||h||_\infty\leq \nu(K),
\]
where we have used the fact that $P_t$ is contractive on $L^\infty (U)$ and that $||h||_\infty\leq 1$. 
Thus,  by \eqref{eq:est-grad h-2},
\[
 \int_U|\nabla h_t|^2 d\mu\leq \nu(K)
\]
for every $t$ so, letting $t\to 0^+$ and using \eqref{eq:est-grad h-1}  
we deduce that
\[
 \int_U|\nabla h|^2 d\mu\leq \nu(K).
\]
Inserting this into \eqref{eq:est-v(K)-1} we  conclude that, for every $\epsilon>0$, $\nu(K)\leq \text{Cap} (K,U) +\epsilon$,
whence, letting $\epsilon \to 0^+$ we finally obtain $\nu(K) \leq \text{Cap} (K,U)$, and since $\nu\in\mathcal{M}^+(K)$ with $||G^U\nu||_\infty\leq 1$ is arbitrary this establishes the inequality
\[
C(K,U)\leq \text{Cap}(K,U).
\]
To prove the reverse inequality, recall that 
\[
\text{Cap(K,U)}= \int_U |\nabla w|^2 d\mu,
\]
where the capacity potential $w$ satisfies 
 \[\left\lbrace
 \begin{array}{l}
 \Delta w = 0 \ \  \text{in } \ U \backslash K, \\
 w = 0 \ \ \text{on } \ \partial U, w = 1 \ \ \text{on } \ K.
 \end{array}\right.
 \]
By the Riesz Decomposition Theorem,  \cite[Thm.4.5.11]{helms}, there is a unique measure $\nu^K \in \mathcal{M}^+(K)$ with $w = G^U\nu^K$. Since $0\leq w\leq 1$, $\nu^K$ is  a test measure for $C(K,U)$ and 
\[
\text{Cap}(K,U) = \nu^K(K)\leq C(K,U),
\]
as required to conclude.
\end{proof}

It was proved by H. Brezis, M. Marcus and A.C. Ponce in \cite[Thm.E.1]{brezis_marcus_ponce} that, in Euclidean space, the standard capacity satisfies, for every compact set $K\subset \mathbb{R}^n$,
\begin{equation}\label{equiv_brezis}
\text{\rm Cap}_{2,1}(K) = 2\text{\rm Cap}(K),
\end{equation}
where $\text{\rm Cap}_{2,1}$ denotes the $(2,1)$-capacity (Laplacian capacity) defined as
\begin{eqnarray*}
\text{Cap}_{2,1}(K) &=& \inf\left\lbrace \int_{\mathbb{R}^n} \vert \Delta u\vert d\mu \colon u \in W^{2,1}_0(\mathbb{R}^n), u \geq 1 \ \text{on} \  K \right\rbrace.
\end{eqnarray*}

In our next result, we extend identity \eqref{equiv_brezis} to Riemannian manifolds. The proof is based on some approximation theorems and is postponed to Appendix \ref{appendix}.

\begin{theorem}\label{BMP-Thm}
Given $U\subset M$ open, for every compact set $K\subset U$ we have
$$\text{\rm Cap}_{2,1}(K,U) = 2\text{\rm Cap}(K,U).$$
\end{theorem}

In view of Theorem \ref{parabolicity_measure_dual}, one may ask whether the $L^1$-capacity is related to 
the standard capacity or, equivalently, to the $(2,1)$-capacity via Theorem \ref{BMP-Thm}. In the next result, we are going to show that they are indeed equivalent. Furthermore, we also show that, for $1<p<\infty$,  the $L^p$-capacity is actually equal to  the $(2,p)$-capacity.

Let $U \subset M$ be a bounded open subset and recall that $W^{2,p}_0(U)$ is the closure of $C^{\infty}_0(U)$ in the norm of $W^{2,p}(M)$. For a compact subset $K \subset U$ we define
\begin{eqnarray*}
\text{Cap}_{2,p}(K, U) &=& \inf\left\lbrace \int_U \vert \Delta u\vert^p d\mu \colon u \in W^{2,p}_0(U), u \geq 1 \ \text{on} \  K \right\rbrace \\
&=& \inf\left\lbrace \int_U \vert \Delta u\vert^p d\mu \colon u \in C^{\infty}_c(U), u \geq 1 \ \text{on} \  K \right\rbrace.
\end{eqnarray*}


\begin{theorem}\label{Cap2,2_C_g2}
Let $K \subset U$ be a compact subset of an open set $U\subset M$. Then for every $1< p<\infty$ we have
$$C_{p}(K,U) = \text{\rm Cap}_{2,p}(K,U) .
$$
For $p=1$ it holds that
\[
C_1( K,U)\leq\mathrm{Cap}_{2,1}(K,U) \leq 2 C_1( K,U),
\]
\end{theorem}
\begin{proof}
The first general inequality follows by taking $u \in C^{\infty}_c(U)$ such that $u\geq 1$ on $K$ and recalling that, by Green's representation, we have $u = G^U(-\Delta u)$. By Remark \ref{general_sign} the function $-\Delta u \in L^p(U)$ is a test for $C_{p}(K,U)$, and thus
$$  C_{p}(K,U) \leq \int_U \vert\Delta u\vert^p d\mu .$$
The result follows by taking infimum over the right hand side.

To prove the two-sided estimate for $p=1$ we first observe that, due to Proposition \ref{prop_limit_regularity}, given $\varepsilon> 0$ we can first find a pre-compact open set $U'$ with smooth boundary such that $K\Subset U' \Subset U$ and $C_1(K, U')\leq C_1(K,U) +\varepsilon/2$ and then, by Proposition \ref{C_g2_test_functions},  a function $0\leq \phi \in C^{\infty}_c(U')$ such that, for some $\delta>0$, $G^{U'}\phi\geq 1+\delta$ on $K$ and
\begin{equation}\label{estim_phi}
\int_{U'}\phi d\mu\leq C_1(K,U') +\varepsilon/2 \leq C_1(K,U) +\varepsilon.
\end{equation}

Let $u =G^{U'}\phi$ on $U'$ extended  by zero to  $U$, so that $u\in C_c^{0,1}(U)$.
Clearly,
\[
\Delta u =
\begin{cases}
-\phi &\text{in } U',\\
0 &\text{in }U\setminus \overline{U'},
\end{cases}
\]
while, integrating by parts shows that, for every $\rho \in C^\infty_c(U\setminus \overline {U''})$, 
where $U''$ is a pre-compact open set with $ \supp \phi\Subset U'' \Subset U'$, we have
\[
\begin{split}
\int_U u\Delta \rho d\mu &=
\int_{U'\setminus\overline{U''}}
\Delta u\rho d\mu
+ \int_{\partial (U'\setminus\overline{U''})} (u\langle \nabla\rho, n \rangle - \rho
\langle \nabla u, n\rangle ) d\mu' \\[0.2cm]
&= -\int_{\partial U'}
\rho
\langle \nabla u, n\rangle d\mu',
\end{split}
\]
where $n$ denotes the outward unit normal and $d\mu'$  the $(n-1)$-dimensional measure induced by $\mu$.
Since $u\geq 0$ in $U'$ and $u=0$ on $\partial U'$, we conclude that the distributional Laplacian of $u$ is the  finite Radon measure given by
\begin{equation}\label{Lapla_dist}
\Delta_{\text{dist}} u = -\phi + \nu,
\end{equation}
where $\nu$ is the positive  measure concentrated on $\partial U'$ defined by
\[
\langle \rho,\nu\rangle =
-\int_{\partial U'}
\rho
\langle \nabla u, n\rangle d\mu',
\]
for every $\rho\in C^\infty_c (U)$.
Its total mass can be easily computed as
$$
\nu(U) =  - \int_{\partial U'} 1\langle \nabla u, n\rangle d\mu'= \int_{U'} \phi d\mu.
$$

Due to the lack of regularity, we consider the smooth $L^1$-approximation $u_r$ of $u$ defined in \eqref{def_u_r} by
\begin{equation*}
u_r(x) = \int_{\tilde U} K_r(x,y) u(y) d\mu(y).
\end{equation*}
It follows from \eqref{def_Deltau_r} that
 \[
    \begin{split} \Delta u_r(x) &= \int_{\tilde U} K_r(y,x) \Delta u(y) d\mu(y)\\&=
    -	\int_{\tilde U} K_r(y,x)\phi(y) d\mu(y) 
    + \int_{\tilde U} K_r(y,x)d\nu(y),
    \end{split}
	\] 
where $\Delta u$ 
is the distributional Laplacian of $u$ and we used \eqref{Lapla_dist}.
Since the two integrals on the right hand side are supported in  $\epsilon$-neighborhoods of $\mathrm{supp} \,\phi$ and $\partial U'$, respectively (see Appendix \ref{appendix}), we obtain

\begin{eqnarray*}
    \int_U \vert \Delta u_r(x)\vert d\mu(x) &=& \int_{\tilde U} \Big\vert -\int_{\tilde U} K_r(y,x)\phi(y) d\mu(y) + \int_{\tilde U} K_r(y,x)d\nu(y)\Big\vert d\mu(x) \\[0.2cm]
&= & \int_{\tilde U} \left( \int_{\tilde U} K_r(y,x)\phi(y) d\mu(y) + \int_{\tilde U} K_r(y,x)d\nu(y)\right) d\mu(x) \\[0.2cm]
&= & \int_{\tilde U} \phi(y) d\mu(y) + \int_{\tilde U} d\nu(y)\\[0.2cm]
&= & 2 \int_{U'} \phi(y) d\mu(y).
\end{eqnarray*}
Therefore, by definition we have
$$\text{\rm Cap}_{2,1}(K,U) \leq \int_U \vert\Delta u_r\vert d\mu = 2 \int_{U'} \phi d\mu .$$
Hence, using \eqref{estim_phi}, taking the infimum over $\phi$ and letting $\varepsilon \rightarrow 0$, we conclude that
$$\text{\rm Cap}_{2,1}(K,U) \leq 2 C_{1}(K,U).$$

To prove that $C_p(K)\geq \operatorname{Cap}_{2,p} (K)$, for $1<p<\infty$,  assume that 
\[
 \operatorname{Cap}_{2,p}(K)^{1/p}=a. 
\]
For each $f\in C(K)$ we define the sublinear functional,
\[
\mathcal P(f) \doteq
\inf\left\{
\|\Delta\varphi\|_{L^p(M)}:
\varphi\in C_c^\infty(M),\ \varphi\ge f \text{ on }K
\right\}.
\]
Note that 
\[
\mathcal P(1)=a.
\]
We define 
$\Lambda_0$  on $\operatorname{span}\{1\}$ by
\[
\Lambda_0(t\cdot 1)=a\cdot t.
\]
It is easy to check that $\Lambda_0(f)\le \mathcal P(f)$ on $\operatorname{span}\{1\}$. By the Hahn-Banach theorem, there exists a linear functional
\[
\Lambda:  C(K)\to \mathbb{R}
\]
such that
\[
\Lambda(f)\le \mathcal P(f)
\quad \forall f\in C(K),
\qquad \text{and }\, 
\Lambda(1)=a.
\]
Moreover,  $\Lambda$ is positive. Indeed, if $f\ge0$, then $\mathcal P(-f)=0$, hence
\[
-\Lambda(f)=\Lambda(-f)\le \mathcal P(-f)=0,
\]
and the conclusion follows. 

By the Riesz-Markov theorem, there exists a positive Radon measure
\[
\nu_o\in\mathcal M^+(K)
\]
such that
\[
\Lambda(f)=\int_K f\,d\nu_o \quad \text{and} \quad \nu_o(K)=\Lambda(1)=a.
\]
Since, for every $\varphi\in C_c^\infty(M)$,  the inequality $\varphi\ge \varphi|_K$ on $K$ is obviously satisfied, we have
\[
\mathcal P(\varphi|_K)\le \|\Delta\varphi\|_{L^p(M)},
\]
whence
\begin{equation}
\label{eq:normestimate}
\left|\int_K \varphi\,d\nu_o\right|
=
|\Lambda(\varphi|_K)|
\le
\mathcal P(\varphi|_K)
\le
\|\Delta\varphi\|_{L^p(M)}.
\end{equation}
It follows that if $\Delta \varphi=\Delta\varphi_1$,
\[
\int_K\varphi d\nu_o=\int_K\varphi_1 d\nu_o,
\]
showing that the functional 
\[
C^\infty_c(M)\ni\varphi\to \int_K \varphi d\nu_o
\]
induces a linear functional defined by 
\[
{F}(\psi)\doteq \int \varphi d\nu_o 
\quad\text{ if }\Delta\varphi= \psi,
\]
on the subspace
\[
\mathcal R\doteq \Delta(C_c^\infty(M))\subset L^p(M),
\]
and  from \eqref{eq:normestimate} we deduce that  $F$ is bounded in $L^p$ with norm $||F||\leq 1$. On the other hand,  the definition of
\[
a=P(1)=\inf\{||\Delta \varphi||_{L^p(M)} \,:\, \varphi \in C_c(M), \,\, \varphi\geq 1 \text{ on } K \},
\]
shows that for every $\epsilon>0$ there exists $\varphi_\epsilon\in C^\infty_c(M)$ such that $\varphi_\epsilon\geq 1$ on $K$ and $||{ \Delta } \varphi_\epsilon||_{L^p(M)} <a+\epsilon$.
Thus,
\[
a=\nu_o(K)=\int_K 1 d\nu_o\leq \int_K \varphi_\epsilon d\nu_o = F(\Delta \varphi_\epsilon)  \leq ||\Delta\varphi_\epsilon||_{L^p(M)}<a+\epsilon,
\]
so that
\[
\frac{F(\Delta \varphi_\epsilon)}{||\Delta \varphi_\epsilon||_{L^p(M)}}\geq \frac{a}{a+\epsilon},
\]
Letting $\epsilon\to 0$ this shows that  $||F||\geq 1$ which, together with the reverse inequality established above, yields $||F||=1$. 

Again by Hahn-Banach theorem, ${F}$ extends to a bounded linear functional
\[
\widehat{F}: L^p(M)\to \mathbb{R} \quad\text {  with } \, \widehat F_{|\mathcal{R}} = F,  \,\,\text{ and } \,\,||\widehat F|| = ||F|| =1,
\]
and then, by  the Riesz representation theorem, there exists $h\in L^q(M)$ such that $\||h||_{L^q(M)}=||\widetilde F||=1$ and
\[
\widehat{F}(u)=\int_M h\,u\,d\mu
\qquad \forall u\in L^p(M).
\]
Finally, given $\varphi\in C_c^\infty(M)$, we  have
\[
\int_K \varphi\,d\nu_o
=
F(\Delta\varphi)
=
\widehat{F}(\Delta\varphi)
=
\int_M h\Delta\varphi\,d\mu.
\]
Thus, in the sense of distributions, $\Delta h=\nu_o $, so that $h=-G\nu_o $. It follows that $G\nu_o \in L^q(M)$ and $||G\nu_o||_{L^q(M)}=1$ whence
\[
C_p(K)^{\frac{1}{p}}= \sup\{\nu(K)\,:\, \nu\in \mathcal{M}^+(K), \,\, ||G\nu||_{L^q(M)}\leq 1\}
\geq \nu_o(K)= a= \operatorname{Cap}_{2,p}(K)^{\frac{1}{p}}.
\]
\end{proof}

\begin{remark}\label{rmk_equiv_C1}
\emph{Combining Theorem \ref{parabolicity_measure_dual} with Theorems \ref{BMP-Thm} and \ref{Cap2,2_C_g2}, we conclude that a compact set $K\subset M$ has zero $L^1$-capacity if and only if the right side of equality \eqref{eq_Cap_C} is zero.}    
\end{remark}

We end this section providing an alternative characterization of the $L^p$-capacity for $1<p<\infty$. First of all, guided by Corollary \ref{cor_extremal_measure}, the capacitary potential of a measure $\nu \in \mathcal{M}^{+}(U)$ is defined by
\begin{eqnarray}\label{def_potential}
V_p^\nu (x) &=& G^U(G^U\nu)^{q-1}(x) \\[0.2cm]
&=& \int_U g^U(x,y)\left(\int_U g^U(z,y)d\nu(z)\right)^{q-1}d\mu(y). \nonumber
\end{eqnarray}
By analogy with the classical potential theory, the quantity
$$ \int_U V_p^\nu d\nu = \int_U (G^U\nu)^{q} d\mu$$
is called the generalized energy (see \cite[Def.2.5.4]{adams_hedberg}).

\begin{proposition}[Thm.2.5.5 in \cite{adams_hedberg}]
Let $K \subset U$ be a compact set of $M$, and $1<p<\infty$. Then
\begin{eqnarray*}
G^U f^K(x) &=& V_p^{\nu^K} \leq 1 \quad \text{for all} \ \ x\in \supp \nu^K, \\[0.2cm]
G^U f^K(x) &=& V_p^{\nu^K} \geq 1 \quad C_{p}\text{-q.e on} \ K.
\end{eqnarray*}
Moreover,
$$C_{p}(K,U) = \max\{ \nu(K) : \nu \in \mathcal{M}^+(K), V_p^\nu (x) \leq 1 \quad \text{for all} \ \ x \in \supp \nu\}.$$
\end{proposition}

Inspired by the variational definition of capacity by means of the generalized energy functional in the classical potential theory, we introduce the following set function which will provide a first alternative definition of $C_{p}$.
\begin{definition}
Let $K \subset M$ be a compact set, and $1<p<\infty$. We define
$$ \text{Cap}_{V_p}(K,U) = \inf \left\lbrace \int_U V_p^\nu d\nu \colon \nu \in \mathcal{M}^+(K), V_p^\nu \geq 1 \ C_{p}\text{-q.e on} \ K \right\rbrace,$$
where $V_p^\nu$ is a capacitary potential for $\nu$.
\end{definition}

In the next result we are going to show the announced equivalence on compact sets.

\begin{theorem}\label{C_H=Cap_H}
Let $K\subset U$ be a compact subset of an open set $U \subset M$, and let $1<p<\infty$. Then
$$ C_{p}(K,U) = \text{Cap}_{V_p}(K,U).$$
\end{theorem}
\begin{proof}
By Corollary \ref{cor_extremal_measure}, there is $\nu^K \in \mathcal{M}^+(K)$ such that
$$ C_{p}(K,U) = \int_U \left(G^U\nu^K\right)^{q}d\mu = \int_U V^{\nu^K}_p d\nu ,$$
where $V^{\nu^K}_p$ is the capacitary potential for $K$. Hence, taking infimum in the right hand side we obtain $\text{Cap}_{V_p}(K,U) \leq C_{p}(K,U)$. For the reverse inequality we may assume $\text{Cap}_{V_p}(K,U) < \infty$, for otherwise there is nothing to prove. So, let $\nu \in \mathcal{M}^{+}(K)$ such that $V_p^{\nu} \in L^1(U,d\nu)$ and $V_p^{\nu} (x) \geq 1$ $C_{p}$-q.e on $K$. Define $h(x) =  (G^U\nu)^{q-1}(x)$. Since $pq-p = q$, $V_p^{\nu} = G^U(G^U\nu)^{q-1}$ and $\supp \nu \subset K$, we can conclude that $h \in L^p_{+}(U)$. Furthermore,
$$ G^U h(x) = G^U(G^U\nu)^{q-1}(x) = V_p^{\nu}(x) \geq 1, \quad C_{p}\text{-q.e on} \ K.$$
Therefore,
$$ C_{p}(K,U) \leq \int_U h^p d\mu = \int_U \left(G^U\nu\right)^{q} d\mu= \int_U V_p^{\nu}d\nu$$
and taking infimum we conclude the proof.
\end{proof}

\section{$L^p$-parabolicity and the $L^{q}$-Liouville property}\label{sec_3}

Let $1\leq q\leq \infty$. A manifold $M$ satisfies the $L^{q}$-Liouville property, briefly $M$ is $L^{q}$-Liouville, if every superharmonic function $u \in L^{q}_+(M)$ is constant. Since parabolic manifolds are trivially $L^{q}$-Liouville, we will restrict ourselves to the non-parabolic case, and let $g(x,y)$ be the positive finite Green kernel of $M$. Our main purpose in this section is to show that, for $1< q< \infty$, the $L^{q}$-Liouville property holds on $M$ if and only if $C_{p}(\overline{E},M) = 0$ for some/every pre-compact open set $E\subset M$,  with $pq = p + q$. The $L^1$-Liouville property and its capacitary description will be the topic of Section \ref{sec_L1}.

Extending the usual definition of parabolicity we set the following

\begin{definition}
Let $1\leq p<\infty$. A manifold $M$ is said to be $L^p$-parabolic if $C_{p}(\overline{E}) = 0$ for every pre-compact open set $E\subset M$.
\end{definition}

With this definition we are going to prove that the  $L^q$-Liouville property is equivalent to the $L^p$-parabolicity. As a first step in this direction, generalizing the characterization for the $L^1$-Liouville property, we show that the $L^q$-Liouville property amounts to the non-integrability of $g(x,\cdot)^q$ in the complement of a compact neighborhood of $x \in M$.

\begin{proposition}\label{prop_equiv_Lq}
The $L^q$-Liouville property holds on $M$ for $1\leq q < \infty$ if and only if $g(x,\cdot)^q$ is non-integrable outside a ball $B_r(x)$ centered at $x \in M$ with radius $r>0$.
\end{proposition}
\begin{proof}
In one direction, if $g(x,\cdot)^q$ is integrable on $M\backslash B_r(x)$, then $u(y) = \min\{g(x,y),1\}$ is a superharmonic function in $L^q_+(M)$ and $M$ is not $L^q$-Liouville. Reciprocally, suppose that  $f \in L^{q}_+(M)$ is a non-identically zero superharmonic function. By the minimum principle, $f$ is positive on $M$. Let $E$ be a pre-compact open set and $x_0 \in M$ such that $B_r(x_0)\subset \overline{E}$. The usual exhaustion argument applied to the Green function centered at $x_0$ gives the existence of a constant $C>0$ such that $g(x_0,y)\leq C f(y)$ for every $y \notin B_r(x_0)$. Then,
\[
\int_{M\backslash B_r(x_0)} g(x_0,y)^{q}d\mu(y) \leq \int_{M\backslash B_r(x_0)} f(y)^{q}d\mu(y) < \infty.
\]
\end{proof}

As a corollary of the above proposition we have the following hierarchy for the $L^q$-Liouville property. We include the case $q=\infty$ for the sake of completeness.

\begin{corollary}\label{hierarchy_Lq}
If $M$ is $L^{q}$-Liouville, then $M$ is also $L^r$-Liouville for every $1\leq r\leq q \leq \infty$.
\end{corollary}
\begin{proof}
Since the Green kernel $g(x,\cdot)$ is harmonic in $M\backslash\{x\}$, and it is obtained as limit of the Dirichlet Green kernel of an exhaustion, we have that
$$ \sup_{M\backslash B_r(x)} g(x,\cdot) \leq \sup_{\partial B_r(x)} g(x,\cdot),$$
for any $B_r(x)\subset M$. The conclusion follows immediately from Proposition \ref{prop_equiv_Lq}.
\end{proof}

Specializing \eqref{def_potential} to measures with densities, we define  the nonlinear Green operator $G^{q}$ acting  on all non-negative measurable functions $f$ by the formula
\begin{equation}\label{Gp'_kernel}
G^{q} f(x) = G(Gf)^{q-1}(x) = \int_M g(x,z)\left(\int_M g(z,y)f(y)d\mu(y)\right)^{q-1}\!\!\!d\mu(z).
\end{equation}

We begin with the following simple comparison result.

\begin{lemma}\label{green_ivariance}
Given $x_0 \in M$ and two concentric balls  $B_0\subset B_1$ centered at $x_0$, there exists a constant $C>0$, depending only on $B_1$, such that
\[
g(x,z) \leq C g(x',z), \ \ \text{for every} \ \ x,x'\in B_0
\ \ \text{and}\ \ z \in B_1^c.
\]
\end{lemma}
\begin{proof}
It is a simple application of Harnack's inequality once it is noticed that for every $z\in B_1^c$, $x\mapsto g(x,z)$ is a positive harmonic function on $B_1$.
\end{proof}

\begin{remark}\label{rmk_general_x_G}
{\rm As a direct consequence of Lemma \ref{green_ivariance} we have that, for every function $f\in C^{\infty}_c(M)$, if $G^{q} f(x_0) = \infty$ for some point $x_0 \in M$, then $G^{q} f(x) = \infty$ for every point $x \in M$. Indeed, assume the contrary and suppose that there exists $x_1 \in M$ satisfying $G^{q} f(x_1) < \infty$. Let  $B_0\subset B_1$  be concentric balls with $x_0, x_1\in B_0$. By the lemma there exists a constant $C>0$ such that
$$ g(x_0,z) \leq C g(x_1,z), \quad \text{for any} \ \ z \in B_1^c.$$
Since
$$ Gf(z) = \int_M g(z,y)f(y)d\mu(y) \in C^{\infty}(M)$$
is bounded on $\overline{B}_1$ and since the Green kernel is locally integrable,  we have
\begin{eqnarray*}
G^{q} f(x_0) &=& \int_M g(x_0,z)\left(\int_M g(z,y)f(y)d\mu(y)\right)^{q-1}d\mu(z) \\[0.2cm]
&=& \int_{\overline B_1} g(x_0,z)(Gf(z))^{q-1} d\mu(z) + \int_{B_1^c} g(x_0,z)(Gf(z))^{q-1}d\mu(z)\\[0.2cm]
&\leq &  \sup_{\overline B_1}(Gf)^{q-1}\vert\vert g(x_0,\cdot)\vert\vert_{L^1(\overline B_1)} + c\int_{B_1^c} g(x_1,z)(Gf(z))^{q-1}d\mu(z) \\[0.2cm]
&<& \infty,
\end{eqnarray*}
which proves the claim.
Indeed, the above argument also shows that  $G^{q}f(x_0)$ is uniformly bounded  in $B_0$ with a bound which depends only on $G^{q}f(x_1)$, on $ \sup_{\overline B_1}(Gf)^{q-1}$ and on
\[
\vert\vert g(x_0,\cdot )\vert\vert_{L^1(\overline B_1)}\leq \sup_{\overline{B}_0} G\varphi<\infty
\]
where $\varphi\in C^\infty_c(M)$ is such that $\mathbf{1}_{\overline{B}_1}\leq \varphi$.}
\end{remark}

In the next result we establish the relation between the non-integrability of $g(x,\cdot)^q$ outside a ball and the explosion of the nonlinear Green operator $G^{q}$ acting on $C^\infty_c(M)$.

\begin{theorem}\label{equiv_G2}
For $q \in [1,\infty)$ the following assertions are equivalent.
\begin{itemize}
\item[i)] $G^{q} \equiv \infty$.
\item[ii)] $G^{q} f(x_0)= \infty$ for some $x_0\in M$, and some $0\leq f \in C^{\infty}_{c}(M)$.
\item[iii)] There exist $x\in M$ and $r>0$ such that $g(x,\cdot) \notin L^{q}(M\backslash B_r(x))$.
\item[iv)] For every $x\in M$ and every $r>0$, $g(x,\cdot) \notin L^{q}(M\backslash B_r(x))$.
\item[v)] $M$ is $L^q$-Liouville.
\end{itemize}
\end{theorem}
\begin{proof}
The equivalence {$\bf iv) \Leftrightarrow v)$} is the content of Proposition \ref{prop_equiv_Lq}.\\ 
{$\bf i) \Rightarrow ii):$} Obvious.\\
{$\bf ii) \Rightarrow iii):$}  Let $ f \in C^{\infty}_c(M)$ be a non-negative function and let  $x_0 \in M$ be such that $G^{q} f(x_0) = \infty$.
Take $r>0$ sufficiently large to have $\mathrm{supp}\,f \subset B_r(x_0)$. By the local integrability of $g(x_0,\cdot)$,
\begin{multline}
\label{finiteness}
\int_{B_{2r}(x_0)}g (x_0,z)\left(\int_M g(z,y)f(y) d\mu(y)\right)^{q-1}d\mu(z)\\= \int_{B_{2r}(x_0)}g (x_0,z)\left(\int_{B_r(x_0)} g(z,y)f(y) d\mu(y)\right)^{q-1}d\mu(z) <\infty.
\end{multline}

On the other hand, by Lemma \ref{green_ivariance}, there exists $C>0$, which depends only on $B_{2r}(x_0)$, such that
$$
g(z,y)\leq C g(z, x_0)\quad \forall z\not\in B_{2r}(x_0) \,\text{ and }\, y\in B_{r}(x_0).
$$
Therefore,
\begin{multline}\label{non_L_2}
\int_{B_{2r}(x_0)^c} g(x_0,z) \left(\int_M g(z,y)f(y) d\mu(y)\right)^{q-1} d\mu(z)
 \\ \leq (C ||f||_{L^1(M)})^{q-1} \int_{B_{2r}(x_0)^c} g(x_0,z) g(z,x_0)^{q-1} d\mu(z)\\
\leq (C ||f||_{L^1(M)})^{q-1} \int_{B_{r}(x_0)^c} g(x_0,z) ^{q} d\mu(z).
\end{multline}
The assumption and \eqref{finiteness} imply that the left hand side in \eqref{non_L_2} is infinite and we conclude that $g(x_0,\cdot) \notin L^{q}(M\backslash B_r(x_0))$.\\
{$\bf iii) \Rightarrow iv):$} Assume that $g(x_0,\cdot)\notin L^{q}(M\backslash B_r(x_0))$ for given $x_0\in M$ and $r_0>0$. Since $g(x_0,\cdot)$ is bounded in every annulus centered at $x_0$ it follows easily that $\mathbf{iv)}$ holds for $x_0 \in M$ and for every $r>0$. Now, given $x \in M$,
Lemma \ref{green_ivariance} applied to the  ball $B_R(x_0)$ with $R=d(x,x_0)+1$ shows that there exists  $C>0$ such that $g(x_0,z) \leq C g(x,z)$ for every $z \in M\backslash B_R(x_0)$. Thus,
\[
\begin{split}
\infty = C^{-q}\int_{M\backslash B_R(x_0)}g(x_0,z)^{q}\,d\mu(z) &\leq \int_{M\backslash B_R(x_0)}g(x,z)^{q}\,d\mu(z)\\ \leq  &\int_{M\backslash B_1(x)}g(x,z)^{q}\,d\mu(z).
\end{split}
\]
{$\bf iv) \Rightarrow i):$} Given a non-identically zero function $0\leq f \in C^{\infty}_c(M)$ let $B_{r}(x_0) \subset \supp f$ such that $\min_{B_{r}(x_0)} f = C_0 >0$.  Using Lemma \ref{green_ivariance} we can compute
\begin{multline*}
\int_{M\backslash B_r(x_0)} g(x_0,z)^{q}d\mu(z) = \int_{M\backslash B_r(x_0)} g(x_0,z)[g(x_0,z)]^{q-1}d\mu(z) \leq  \\[0.2cm]
 \int_{M\backslash B_r(x_0)} g(x_0,z)\left[\frac{c}{C_0\mu(B_r(x_0))}\int_{B_{r}(x_0)} g(z,y)f(y) d\mu(y)\right]^{q-1}d\mu(z),
\end{multline*}
which yields
$$\int_{M\backslash B_r(x_0)} g(x_0,z)^{q}d\mu(z) \leq \left(\frac{c}{C_0\mu(B_r(x_0))}\right)^{q-1} G^{q}f(x_0).$$
Thus, by our assumption we have $ G^{q}f(x_0) = \infty$ and the conclusion follows by Remark \ref{rmk_general_x_G}.
\end{proof}

\begin{remark}\label{rmk_outside_U}
{\rm We observe for future use that the above proof actually shows that $G^{q} \equiv \infty$ if and only if, for some/every $x\neq y$ in $M$ and some/every pre-compact open set $U$ containing $x$ and $y$,
$$\int_{M\backslash U}g(x,z)g(z,y)^{q-1}d\mu(z) = \infty.$$}
\end{remark}

\begin{remark}\label{rmk_uniform_L2}
{\rm If
$G^{q} < \infty$, then Theorem \ref{equiv_G2} can be used to prove that for every $x \in M$, and every $r>0$, there exists a constant $C>0$, which depends only upon $B_{2r}(x)$, such that
\begin{eqnarray*}
\int_{M\backslash B_{2r}(x)} g(y,z)^{q} d\mu(z) < C \quad \text{for all } \ y \in B_r(x).
\end{eqnarray*}
Indeed, by assumption  and Theorem \ref{equiv_G2},
\[
\int_{M\backslash B_{2r}(x)}g(x,z)^{q} d\mu(z)<\infty,
\]
and the conclusion follows again from Lemma~\ref{green_ivariance}.}
\end{remark}

\begin{proposition}\label{equiv_lemma_faraji_grig}
$M$ is $L^q$-Liouville, for $q \in [1,\infty)$, if and only if for every positive superharmonic function $f \colon M \rightarrow \mathbb{R}$, we have $G[f^{q-1}](x) \equiv \infty$ for some/every $x \in M$.
\end{proposition}
\begin{proof}
Assume that $M$ is $L^q$-Liouville and let $f$ be a positive superharmonic function on $M$. Fix $y \in M$ and let $U$ be a pre-compact neighborhood of $y$. As in Proposition \ref{prop_equiv_Lq} there is a constant $c>0$ such that
$$ g(z,y) \leq c f(z), \ \ \text{for any} \ \ z \notin U.$$
So that, raising to power $q-1$ and integrating we obtain
$$ G[f^{q-1}](x) \geq c^{-1}\int_{M\backslash U} g(x,z)g(z,y)^{q-1}d\mu(z) = \infty.$$
The conclusion then follows by Remark \ref{rmk_outside_U}. Conversely, let $\delta>0$ be a small constant such that $m = \inf_{B_{\delta}(y)}g(y,\cdot) >  \sup_{\partial B_1(y)}g(y,\cdot)$ and define $f = \min\{m,g(y,\cdot)\}$. Then $f$ is a positive superharmonic function in $M$ and
$$ \int_M g(x,z)g(z,y)^{q-1}d\mu(z) \geq G[f^{q-1}](x) = \infty,$$
and the conclusion follows by Theorem \ref{equiv_G2}.
\end{proof}

\begin{remark}\label{rem_lemma_faraji_grig}
In Proposition \ref{equiv_lemma_faraji_grig} we can use positive functions defined and harmonic in a complement of a compact set. In this case, if we fix a compact set $E\subset M$, we then have
$$ G^q \equiv \infty \quad \text{if and only if} \quad \int_{M\backslash E} g(x,y)h(y)^{q-1} d\mu(y) = \infty.$$
\end{remark}

As an application of Theorem \ref{equiv_G2} we now extend, from the classical case $q=\infty$, the characterization of the $L^q$-Liouville property in terms of the existence of pre-compact open sets with zero $L^{p}$-capacity  to the range $1< q \leq \infty$, with $p$ and $q$ H\"older conjugate exponents. The case $p=1$ is included in the statement only for completeness.

\begin{theorem}\label{characterization_biparb_capacity}
Let $M$ be a manifold and $p \in [1,\infty)$. Then $C_{p}(\overline{E}) = 0$ for some/every pre-compact open subset $E \subset M$ if and only if $M$ is $L^q$-Liouville. Therefore, $M$ is $L^p$-parabolic if and only if it is $L^q$-Liouville. 
\end{theorem}
\begin{proof}
By Theorem \ref{equiv_G2} we assume $G^{q}\equiv \infty$ and let $E\subset M$ be a pre-compact open set. Since the $C_{p}$-capacity is monotone and subadditive and
$E$ is pre-compact, it suffices to prove that $C_p(B_1(y_0)) = 0$ for any $y_0 \in M$. As in the proof of Theorem \ref{equiv_G2}, by Lemma \ref{green_ivariance} there exists $C>0$, depending only on $B_2(y_0)$, such that $g(y_0,z) \leq Cg(y,z)$ for all $y \in B_1(y_0)$ and for all $z \notin B_2(y_0)$. Given any measure $0\ne \nu \in \mathcal{M}^+(B_1(y_0))$ we estimate
\begin{eqnarray*}
\vert\vert G\nu\vert\vert_{L^{q}(M)} &\geq & \left(\int_{M\backslash B_2(y_0)} \left( \int_{B_1(y_0)} g(y,z)d\nu(y)\right)^{q} d\mu(z)\right)^{\frac{1}{q}} \\[0.2cm]
& \geq & C^{-1}\nu(B_1(y_0)) \left(\int_{M\backslash B_2(y_0)} g(y_0,z)^{q} d\mu(z)\right)^{\frac{1}{q}} \\[0.2cm]
&=& \infty,
\end{eqnarray*}
where in the last equality we have applied Theorem \ref{equiv_G2} item ${\bf iv)}$. Thus, if $\vert\vert G\nu\vert\vert_{L^{q}(M)}\leq 1$, $\nu\equiv 0$ and by Theorem \ref{dual_cap_thm} and Remark~\ref{rmk_equiv_C1} we have $C_{p}(B_1(y_0)) = 0$.

For the converse, we are going to show that if $G^{q} < \infty$ then $C_{p}(\overline E)>0$ for every pre-compact open subset $E$ with $\mu(\overline{E})>0$. So, assume $G^{q} < \infty$ and that there exists some pre-compact open subset $E$ such that $C_{p}(\overline E)=0$. For every $x$, let $F_x=B_1(x)\cap \overline E$ whenever the right hand side is not empty.

Notice first of all that since $G^{q} < \infty$, then $G\chi_{F_x}\in L^{q}(M).$  Indeed, if $\psi\in C_c^\infty(M)$ is non-negative and $\psi=1$ on $B_1(x)$, then
\[
\begin{split}
\int_{B_2(x)} (G\chi_{F_x})^{q}(y)d\mu(y)&\leq  \int_{B_2(x)}\left(\int_Mg(y,z)\psi(z)d\mu(z)\right)^{q} d\mu(y)\\
&\leq\mu(B_2(x)) \sup_{\overline B_2(x)} |G\psi|^{q} <\infty.
\end{split}
\]
On the other hand,  by Remark \ref{rmk_uniform_L2}, there exists a constant $C$ depending only on $B_2(x)$ such that
\[
\int_{M\backslash B_2(x)} g(y,z)^{q} d\mu(y)\leq C, \quad \forall z\in B_1(x),
\]
and therefore,  by Minkowski's integral inequality,
\[
\begin{split}
\left[\int_{M\backslash B_2(x)} (G\chi_{F_x} (y))^{q} d\mu(y) \right]^{\frac{1}{q}}&=\left[\int_{M\backslash B_2(x)}  \left(\int_{F_x} g(y,z)d\mu(z)\right)^{q} d\mu(y)\right]^{\frac{1}{q}} \\
&\leq \int_{F_x} \left(\int_{M\backslash B_2(x)} g(y,z)^{q} d\mu(y)\right)^{\frac{1}{q}}d\mu(z) \\
&\leq  C^{\frac{1}{q}} \mu(B_1(x)),
\end{split}
\]
and our claim follows.

Next, since $C_{p}(\overline E)=0$, by definition of capacity, for every $\epsilon>0$ there exists a function  $f\in L^p_+(M)$ such that $Gf\geq 1$  on $\overline E$ and
\[
\int_M f^p d\mu<\epsilon^p.
\]
We can therefore estimate
\begin{equation*}
\begin{split}
\mu(F_x)&\leq \int_{F_x} Gf(y)d\mu(y)\\
&= \int_M\chi_{F_x}(y)Gf(y)d\mu(y) \\
&= \int_{M\times M} g(y,z)f(z)\chi_{F_x}(y)  d\mu(z)d\mu(y) \\
&=\int_M f(z)G\chi_{F_x}(z)d\mu(z)\\
&\leq ||f||_{L^p(M)} ||G\chi_{F_x}||_{L^{q}(M)}<\epsilon ||G\chi_{F_x}||_{L^{q}(M)},
\end{split}
\end{equation*}
and, letting $\epsilon\to 0$, we conclude that
\[
\mu(F_x)=\mu(B_1(x)\cap E)=0
\]
for every $x$ such that $B_1(x)\cap \overline{ E}\ne \emptyset$. Since
$\overline{E} $ is compact, it can be covered by a countable family of balls $B_1(x_k)$. Hence,
$\overline{ E}=\cup_k F_{x_k}$ and by subadditivity we conclude that $\mu(\overline{ E})=0$, as required. 
\end{proof}

\begin{remark}
{\rm In the above proof, since $M$ is second countable, we have proved that if $G^{q} < \infty$ then $C_{p}(K)=0$ implies  $\mu(K)=0$ for every compact set $K$. It means that, in this case, the capacity $C_{p}$ gives a more refined measure than the Lebesgue one. Furthermore, if $G^{q} = \infty$ then $C_{p}(E)=0$ for all subsets $E\subset M$. }
\end{remark}

\begin{corollary}\label{cor_inv_zero_cap}
There is a pre-compact open subset $E \subset M$ with $C_{p}(E) > 0$ if and only if $C_{p}(F) > 0$ for every pre-compact open set $F \subset M$.
\end{corollary}

As a notable consequence of Theorem \ref{characterization_biparb_capacity} and Corollary \ref{hierarchy_Lq} we can derive the hierarchy of the $L^p$-parabolicity.

\begin{corollary}\label{cor_hier}
If $M$ is $L^{p}$-parabolic, then $M$ is also $L^s$-parabolic for every $1\leq p\leq s < \infty$.
\end{corollary}

\label{subsec_bip}

In the recent paper \cite{faraji_grigoryan} the authors introduced the concept of biparabolicity for a Riemannian manifold in terms of a Liouville type condition for the bilaplacian operator. A function $u \in C^4(M)$ is said to be bi-superharmonic if $\Delta u\leq 0$ and $\Delta^2 u \geq 0$.
\begin{definition}
A manifold $M$ is biparabolic if any positive bi-superharmonic function $u \in C^4(M)$ is harmonic, $\Delta u = 0$.
\end{definition}

One of the main results in \cite{faraji_grigoryan}, Theorem 3.1, is a characterization of the biparabolicity of $M$ by the explosion of the Green operator $G^2$ defined in \eqref{Gp'_kernel}. Thus, Theorems \ref{equiv_G2} and \ref{characterization_biparb_capacity} give rise to the following application.

\begin{corollary}
A manifold $M$ is biparabolic if and only if $M$ is $L^2$-parabolic, equivalently, $L^2$-Liouville.
\end{corollary}

In this direction, we note that Proposition \ref{equiv_lemma_faraji_grig} and Remark \ref{rem_lemma_faraji_grig} extend \cite[Lem.3.2]{faraji_grigoryan}.

\section{A capacitary approach to the $L^1$-Liouville property}\label{sec_L1}

Since the positive minimal Green kernel is locally integrable, according to Proposition \ref{prop_equiv_Lq}, the $L^1$-Liouville property on a manifold $M$ is equivalent to the non-integrability of $g(x,\cdot)$ (see also \cite{grigoryan_89,grigoryan_bams}). From this one can easily deduce that every stochastically complete manifold satisfies the $L^1$-Liouville property and, on model manifolds, the two properties are equivalent. In general this equivalence is not true in any dimension (see \cite{bessa_pigola_setti,pessoa_pigola_setti}). Note that, if $M$ is not $L^1$-Liouville, the function
$$ E(x) = \int_M g(x,y)d\mu(y),$$
which is usually referred to as the mean exit time of $M$, is a positive solution to the equation
$$ \Delta E + 1 = 0 .$$

As in the previous subsection we are going to restrict ourselves to the non-parabolic case. Since there is no duality when  $q = 1$, we will define the relevant capacity via positive measures.

\begin{definition}\label{def_c_g_infinity}
Given $K \subset U$ a compact domain of an open set $U\subseteq M$, we define
$$ C_{\infty}(K,U) = \sup \{ \nu(K) \colon \nu \in \mathcal{M}^+(K), \vert\vert G^U\nu\vert\vert_{L^1(U)} \leq 1\}.$$
When $U=M$, we simply write $C_{\infty}(K)$ in place of $C_{\infty}(K,M)$. A manifold $M$ is then said to be $L^\infty$-parabolic if $C_{\infty}(K) = 0$ for every compact subset $K\subset M$.
\end{definition}
The set function $C_{\infty}$ is monotone increasing in $K$, and monotone decreasing in $U$. 
In the next theorem we provide a description of the extremal measures for the $C_{\infty}$-capacity which in turn will easily yield the required capacitary characterization of the $L^1$-Liouville property and its equivalence with the $L^\infty$-parabolicity defined above.

\begin{theorem}\label{L1_liouville_cap}
Let $K$ be a compact domain of an open set $U \subseteq M$. There exists an extremal measure $\nu^K \in \mathcal{M}^+(K)$ such that
\begin{equation*}
C_{\infty}(K,U) = \nu^K(K).
\end{equation*}
Moreover, $\nu^K = \left(\min_{K} E^U\right)^{-1}\delta_{x_0}$ where $\delta_{x_0}$ is the Dirac measure centered at $x_0 \in \partial K$, and $E^U(x_0) = \min_{K} E^U$. 
In particular, a manifold $M$ is $L^1$-Liouville if and only if $C_{\infty}(K) = 0$ for some/every compact set $K$.
\end{theorem}
\begin{proof}
We first consider the case where, for some/every point $x \in U$, we have
$$E^U(x) = \int_U g^U(x,y)d\mu(y) < \infty.$$
In this case, $E^U$ is a smooth function, and for any test measure $\nu \in \mathcal{M}^+(K)$, we have
$$  1\geq \int_U G^U\nu(y)d\mu(y) \geq \nu(K)\min_{K} E^U.$$
Hence,
$$C_{\infty}(K,U) \leq \frac{1}{\min_K E^U}\cdot$$
On the other hand, if $E^U(x_0) = \min_{K} E^U$ for some $x_0 \in \partial K$ (here we use that $E^U$ is superharmonic), then the Dirac measure $\nu_0 = (E^U(x_0))^{-1} \delta_{x_0} \in \mathcal{M}^+(K)$ satisfies
$$ \int_U G^U\nu_0(y)d\mu(y) = \frac{1}{E^U(x_0)}\int_U g^U(x_0,y)d\mu(y) = 1.$$
Moreover,
$$\nu_0(K) = \frac{1}{E^U(x_0)} \geq C_{\infty}(K,U),$$
that is, $\nu_0$ is the extremal measure for the capacity. Now, assume that
\[
E^U(x) = \int_U g^U(x,y)d\mu(y)=\infty
\]
for some/every $x \in U$, and let $\nu\in \mathcal{M}^+(K)$ be such that $\nu(K)>0$. Then
\[
\begin{split}
\int_U G^U\nu(y)d\mu(y) &= \int_U  d\mu(y)\int_U g^U(x,y)d\nu(x)\\[0.2cm]
&=\int_K d\nu(x)\int_U g^U(x,y)d\mu(y) = \infty,
\end{split}
\]
showing that
\[
C_{\infty}(K,U)=0.
\]
The second statement is now a consequence of the already noted fact that $M$ is $L^1$-Liouville if and only if $E \equiv \infty$.
\end{proof}

As in Proposition \ref{prop_limit_regularity} the monotonicity of $C_{\infty}(K,U)$ with respect to $U$ gives rise to the following identity
$$ C_{\infty}(K,U) = \lim_{n\rightarrow \infty}  C_{\infty}(K,U_n),$$
where $\{U_n\}$ is a pre-compact exhaustion of any open set $U\subset M$, such that $K \subset U_1$. It is easy to see that such a limit does not depend on the exhaustion and satisfies
$$ C_{\infty}(K,U) \leq \lim_{n\rightarrow \infty}  C_{\infty}(K,U_n).$$
For the other inequality we consider $g_{n}(x,y)$ the Dirichlet Green function associated to $U_n$, and extend $g_n(x,y)$ by zero on $U\backslash U_n$. For every $y \in U$, and $x \in U$ fixed, we know that $g_n(x,y) \nearrow g^U(x,y)$. It follows that
$$E_n(x) \doteq \int_U g_n(x,y)d\mu(y) \nearrow \int_U g^U(x,y)d\mu(y) \doteq E^U(x)$$
for any $x\in U$. As before, if $E^U(x) < \infty$ then $E^U$ is a smooth function, and by Dini's theorem we have $E_n \nearrow E^U$ uniformly on compact sets, in particular on $K$. Now, by uniform convergence $\min_{K}E_n \rightarrow \min_K E^U$ and the conclusion follows from Theorem \ref{L1_liouville_cap}. If $E^U(x) = \infty$, since $E_n \rightarrow E^U$ locally uniformly it follows that  $\min_{K}E_n \rightarrow \infty$.

We may extend the definition of the set function $C_{\infty} (\, \cdot \, ,U)$ to arbitrary sets via the standard min-max procedure. First, if $G\subset U$ is an open subset we let
$$C_{\infty}(G,U) \doteq \sup \{C_{\infty}(K,U) \colon K\subset G \ \ \text{is compact}\},$$
and, for every $F \subset U$, we define
$$C_{\infty}(F,U) \doteq \inf \{C_{\infty}(G,U) \colon G\supset F \ \ \text{is open}\}.$$

For any open subset $F \subset U$ we can show that
\begin{equation}\label{inf_eq_G_U}
C_{\infty}(F,U) = \left(\inf_{F} E^U\right)^{-1}.
\end{equation}
Indeed, if $G\subset U$ is open, it follows easily from the properties of $\inf\!/\!\sup$ that, for every $G\subset U$ open,
$$ \inf_G E^U = \inf_{K\subset G} \inf_K E^U,$$
where $K\subset G$ is compact, whence
$$ C_{\infty}(G,U) = \sup_{K\subset G} C_{\infty}(K,U) = \left(\inf_{K \subset G} \inf_K E^U\right)^{-1} = \left(\inf_G E^U\right)^{-1}.$$
Similarly, using the fact that, for every $F\subset U$,
$$ \inf_F E^U = \sup_{G\supset F} \inf_G E^U,$$
where $G\subset U$ is open, we deduce that
$$C_{\infty}(F,U) = \inf_{G\supset F} C_{\infty}(G,U) = \left(\inf_F E^U \right)^{-1}.$$

We now collect the properties of the set function $C_{\infty}$ in the following proposition.

\begin{proposition}\label{prop_prop_C_g_infty}
$C_{\infty}$ enjoys the following properties.
\begin{itemize}
\item[1)] Given $F_1 \subset F_2 \subset U$ subsets of an open set $U\subset M$, we have
$$C_{\infty}(F_1,U) \leq C_{\infty}(F_2,U).$$
\item[2)] Given $U_1 \subset U_2 \subset M$ open sets such that $F\subset U_1$, it holds
$$C_{\infty}(F,U_1) \geq C_{\infty}(F,U_2).$$
\item[3)] Given $F_1, F_2$ arbitrary subsets of an open set $U\subset M$, there holds
$$ C_{\infty}(F_1 \cup F_2,U) \leq C_{\infty}(F_1,U) + C_{\infty}(F_2,U) - C_{\infty}(F_1 \cap F_2,U).$$
\item[4)] Let $K_i$ be a decreasing sequence of compact sets contained in an open set $U\subset M$, then
$$ C_{\infty}(\displaystyle \cap_{i=1}^{\infty} K_i ,U) = \lim_{i \rightarrow \infty}  C_{\infty}(K_i ,U).$$
\item[5)] Given an increasing sequence of arbitrary sets $F_i$ contained in an open set $U\subset M$ there holds
$$ C_{\infty}(\cup_i F_i,U) = \lim_{i \rightarrow \infty} C_{\infty}(F_i,U).$$
\end{itemize}
\end{proposition}
\begin{proof}
Items $1)$ and $2)$ are an easy consequence of \eqref{inf_eq_G_U}. For item $3)$, let $F_1, F_2 \subset U$ be arbitrary sets. Equation \eqref{inf_eq_G_U} yields
\begin{eqnarray*}
C_{\infty}(F_1 \cup F_2,U) &=& \left(\inf_{F_1 \cup F_2} E^U\right)^{-1} \\
&=& \left[\min\{ \inf_{F_1} E^U, \inf_{F_2} E^U\}\right]^{-1} \\
&\leq & \left(\inf_{F_1} E^U\right)^{-1} + \left(\inf_{F_2} E^U\right)^{-1} - \left(\inf_{F_1 \cap F_2} E^U\right)^{-1} \\[0.1cm]
&=& C_{\infty}(F_1,U) + C_{\infty}(F_2,U) - C_{\infty}(F_1 \cap F_2,U).
\end{eqnarray*}
Item $4)$ will follow from the following identity
$$ \min_{\cap_n K_i} E^U = \lim_{i \rightarrow \infty} \min_{K_i} E^U.$$
Indeed, by monotonicity it is clear that the limit exists, and the left hand side is no less than the right hand side. On the other hand, for every $i$ let $x_i \in K_i$ be the point minimizing $E^U$ on $K_i$. By passing to a subsequence we may assume that $x_i \rightarrow x \in \cap_{i}K_i$. Hence,
$$\min_{\cap_{i}K_i} E^U \leq E^U(x) = \lim_{i \rightarrow \infty} E^U(x_i) = \lim_{i \rightarrow \infty} \min_{K_i} E^U.$$
To prove item $5)$ let us set $F = \cup_i F_i$. The claim follows from the fact that $\inf_F E^U=\lim_i \inf_{F_i}E^U$ which in turn is easily seen as follows. On the one hand, monotonicity shows that the limit exists and the left hand side is less than or equal to the right hand side. On the other hand, for every $\epsilon>0$, there exists $x_\epsilon\in F$ such that $E^U(x_\epsilon)< \inf _F E^U + \epsilon$. Since $x_\epsilon\in F_i$ for large enough $i$, for such $i$'s
$
\inf_{F_i} E^U\leq E^U(x_\epsilon)\leq \inf_F E^U +\epsilon,
$
whence letting $i\to \infty$ and $\epsilon \to 0$ we conclude that
$\lim_i \inf_{F_i} E^U\leq \inf_F E^U$ and our claim follows.
\end{proof}

The properties listed in Proposition \ref{prop_prop_C_g_infty} show that $C_{\infty}(\cdot,U)$ is a Choquet capacity, and thus, all Borel sets $F\subset U$ are capacitable (cf. \cite[Thm.2.3.11]{adams_hedberg}), that is,
\begin{eqnarray*}
C_{\infty}(F,U) &=& \inf\{C_{\infty}(G,U) \colon G\supset F, G  \ \text{open}\}\\[0.2cm]
&=& \sup\{C_{\infty}(K,U) \colon K\subset F, K  \ \text{compact}\}.
\end{eqnarray*}

\section{Volume conditions for the $L^{p}$-parabolicity} \label{sec_4}

In this section we first provide a sufficient pointwise volume condition for the validity of the $L^{p}$-parabolicity of a complete manifold $M$ in the range $1<p\leq 2$, which extends the results obtained in \cite{faraji_grigoryan} for $p=2$, and is compatible with the known volume conditions for the usual parabolicity in the limit $p \to 1$. For manifolds with non-negative Ricci curvature we obtain a sufficient integral condition in the whole range $1<p<\infty$ that implies an improved volume condition. Moreover, this integral condition turns out to be always valid, and essentially sharp, for general model manifolds.



In what follows we are going to denote by $V(o,r)$, or simply $V(r)$, the volume of the ball $B_r(o)$ and by $S(o,r)$, or simply $S(r)$, the area of the sphere $\partial B_r(o)$, where $o \in M$ is a given reference point.

\begin{theorem}\label{vol_cond_general}
Let $M$ be a complete manifold and let $1<p\leq 2$. Assume that, for some $o \in M$ and sufficiently large $r$, it holds
\begin{equation}\label{vol_cond_almost}
V(r) \leq C \frac{r^{2p}}{\log r},
\end{equation}
for some constant $C>0$. Then $M$ is $L^{p}$-parabolic.
\end{theorem}
\begin{proof}
From Theorems \ref{equiv_G2} and \ref{characterization_biparb_capacity} it is sufficient to prove that $G^{q}\varphi \equiv \infty$ for some $0\leq \varphi \in C^{\infty}_c(M)$, and $pq = p + q$. Recall that the Green kernel is given by
$$g(x,y) = \int_0^\infty p_t(x,y)dt,$$
where $p_t(x,y)$ is the heat kernel of $M$. By Jensen's inequality and the semi-group property of the heat kernel we obtain
\begin{eqnarray*}
G^{q}\varphi(o) 
&=& \int_0^\infty dt\int_M \left(\int_M g(y,z)\varphi(z)d\mu(z)\right)^{q-1} p_t(o,y)d\mu(y) \\[0.2cm]
&\geq & \int_0^\infty \left(\int_M p_t(o,y)\int_M g(y,z)\varphi(z)d\mu(z) d\mu(y)\right)^{q-1} dt \\[0.2cm]
&=& \int_0^\infty \left( \int_0^\infty P_{t+s}\varphi(o) ds\right)^{q-1} dt \\[0.2cm]
&=& \int_0^\infty \left(\int_t^\infty P_\tau \varphi(o) d\tau\right)^{q-1} dt,
\end{eqnarray*}
where $P_t$ is the heat operator acting on $C^\infty_c(M)$ by
$$P_t\varphi(x) = \int_M p_t(x,y)\varphi(y)d\mu(y).$$

To estimate $P_t\varphi(o)$ from below we argue as in \cite{faraji_grigoryan}: let $\text{supp} \,\varphi\subset B_R(o)$.
It is proved in \cite{coulhon_grigoryan} (see also \cite[Thm.16.5]{grigoryan_book}) that a polynomial volume estimate of the form $V(r)\leq Cr^{\nu}$, for $r\geq r_0$ implies that there exist constants $t_0=t_0(r_0)$ and $K=K(o,r_0,C, \nu)$ such that the following  heat kernel diagonal lower bound
\[
p_t(o,o)\ge \frac{1}{4}  [V(\sqrt{Kt\log t})]^{-1}
\]
holds for every $t\geq t_0=t_0(r_0)$. Together with the
local parabolic H\"arnack inequality in \cite{saloff-coste,saloff-coste1} we obtain
\begin{equation}\label{final_estimate_vol}
\begin{split}
 P_\tau \varphi(o) &=\int_{B_R(o)} p_\tau (o,y)\varphi(y)
 \geq  c\int_{B_R(o)} p_\tau(o,o)\varphi(y)  \\&\geq
 c ||\varphi||_{L^1}  \tau^{-\frac{q}{q-1}} (\log \tau)^{-\frac{1}{q-1}} \quad \text{for} \ \ \tau \geq \tau_0,
 \end{split}
\end{equation}
where $c>0$ depends on $r_0, C,\nu,  R$ and on the geometry of $B_R(o) \subset M$. Inserting this into the above inequality we conclude that
\begin{eqnarray*}
G^{q}\varphi(o) &\geq & c  \int_{\tau_0}^\infty \left(\int_t^\infty \tau^{-\frac{q}{q-1}} (\log \tau)^{-\frac{1}{q-1}} d\tau\right)^{q-1} dt \\[0.2cm]
& \asymp & \int_{\tau_0}^\infty \frac{dt}{t \log t} = \infty.
\end{eqnarray*}
\end{proof}

\begin{remark}\emph{
A sharp estimate in the range $1<p \leq 2$ can be obtained if we assume that $M$ satisfies the following heat kernel diagonal lower bound
\begin{equation}\label{upper_hk}
p_t(o,o) \geq \frac{c}{V(\sqrt{t})},
\end{equation}
for some $c>0$ and for all $t>0$. In this case, we can replace \eqref{vol_cond_almost} by the sharp inequality
\begin{equation*}
V(r) \leq C r^{2p}(\log r)^{p-1},
\end{equation*}
to obtain \eqref{final_estimate_vol}. Thus, the result follows as in the proof of Theorem \ref{vol_cond_general}. For instance, \eqref{upper_hk} holds if $M$ satisfies the volume doubling condition and a comparable diagonal heat kernel upper estimate or, equivalently, the relative Faber-Krahn inequality (see \cite[Thm.15.21 and Cor.16.7]{grigoryan_book}). 
}
\end{remark}

Let us now consider the class of geodesically complete Riemannian manifolds with non-negative Ricci curvature. In this class, P. Li and S.-T. Yau \cite{li-yau} established the following Green function estimate
\begin{equation}\label{li-yau}
C^{-1} \int_r^\infty \frac{t\,dt}{V(x,t)} \leq g(x,y) \leq C  \int_r^\infty \frac{t\,dt}{V(x,t)},
\end{equation}
where $r = d(x,y)$ and $C>0$. 
A characterization of $L^p$-parabolicity in the whole range $1<p<\infty$ can be deduced from Theorems~\ref{equiv_G2} and \ref{characterization_biparb_capacity}.

To state the next results, fix some $x\in M$ and set $V(r)=V(x,r)$.

\begin{proposition}\label{prop_ricc} Let $M$ be a complete Riemannian manifold with $\text{Ric}\geq 0$, and let $1<p<\infty$. Then $M$ is $L^p$-parabolic if and only if
\begin{equation}\label{int_vol_1}
\int^\infty  \left(\int_r^\infty \frac{t\,dt}{V(t)}\right)^{\frac{p}{p-1}}V'(r)dr = \infty.
\end{equation}
\end{proposition}
\begin{proof}
By Theorem~\ref{equiv_G2}, $M$ is $L^{q}$-Liouville with $q=\frac{p}{p-1}$ if and only if for some/all $\varepsilon >0
$
\begin{equation*}
\int_{M\backslash B_{\varepsilon}(x)}g(x,y)^{q}d\mu(y) =\infty .
\end{equation*}%
Integrating the Li-Yau Green function estimate \eqref{li-yau} we obtain%
\[
\int_{M\backslash B_\varepsilon(x)} g(x,y)^{q}d\mu(y) \asymp  \int_\varepsilon^\infty  \left(\int_r^\infty \frac{t\,dt}{V(t)}\right)^{q}V'(r)dr.
\]
Finally, by Theorem~\ref{characterization_biparb_capacity}, $M$ is $L^{p}$-parabolic if and only if $M$ is $L^{q}$-Liouville.
\end{proof}

As a first consequence, we can obtain a sharp sufficient integral volume condition to the validity of the $L^p$-parabolicity.

\begin{corollary}\label{Vp_ricc}
Let $M$ be a complete Riemannian manifold with $\text{Ric}\geq 0$, and let $1<p<\infty$. If
\begin{equation}
\int^{\infty }r\left( \int_{r}^{\infty }\frac{tdt}{V(t)}\right) ^{\frac{1}{%
p-1}}dr=\infty,  \label{Vp}
\end{equation}%
where \(V(t)=V(x,t) \) for some \(x\in M,\) then $M$ is $L^p$-parabolic.
\end{corollary}
\begin{proof}
Since every parabolic manifold is $L^p$-parabolic for every $1<p<\infty$, let us assume that $M$ is non-parabolic. By the Li-Yau estimate \eqref{li-yau} the inner integral in \eqref{Vp} is finite for any $r>0$. Integration by parts yields
\begin{align*}
\int^\infty_1 \left( \int_r^\infty \frac{tdt}{V(t)}\right)^\frac{p}{p-1}V'(r)dr
\geq&-\int^\infty_1 V(r)\frac{d}{dr}\left( \int_r^\infty \frac{tdt}{V(t)}\right)^\frac{p}{p-1}dr\\
&- V(1) \left(\int_1^\infty \frac{tdt}{V(t)}\right)^\frac{p}{p-1}\\
=&  \frac{p}{p-1}\int^\infty_1 r\left( \int_r^\infty \frac{tdt}{V(t)}\right)^\frac{1}{p-1}dr \\
 &- V(1) \left(\int_1^\infty \frac{tdt}{V(t)}\right)^\frac{p}{p-1},
\end{align*}
and the conclusion follows by Proposition~\ref{prop_ricc}.
\end{proof}

\begin{remark}\emph{
For example, if
\begin{equation*}
V(r)\leq Cr^{2p}\left( \log r\right) ^{p-1}
\end{equation*}%
then \eqref{Vp} is satisfied. In particular, if $M$ has dimension $n\geq 2$ and satisfies $\text{Ric}\geq 0$, by the Bishop-Gromov volume comparison theorem $V(r) \leq C r^n$, for some $C>0$, and therefore $M$ is $L^p$-parabolic for every $p\geq n/2$.}
\end{remark}


Motivated by this result we make the following

\begin{conjecture}\label{conjecture}
Let $M$ be a complete Riemannian manifold, and let $p \in (1,\infty)$. If
\[
\int^\infty r \left( \int_r^\infty \frac{tdt}{V(t)}\right)^{\frac{1}{p-1}}dr = \infty,
\]
then $M$ is $L^p$-parabolic.
\end{conjecture}

%

We conclude this section by observing that Conjecture \ref{conjecture} holds also for the class of model manifolds $(M_\sigma,ds^2)$, where $M_\sigma = \mathbb{R}^n$ and, in polar coordinates $(r,\theta)$,
$$ds^2= dr^2+ \sigma^2(r) d\theta^2$$
for a smooth, positive function $\sigma$ on $(0,\infty)$.
In this case, the Green kernel with pole at $o=0$ is radial and it is given by
 \[
 g(o, x)=\int_r^{\infty} \frac{dt}{S(t)}, \quad \text{if }  x=(r,\theta),
 \]
where $S(r)=c_m \sigma^{m-1}(r)$ is the Riemannian volume of the sphere $\partial B_r(o)$ centered at the pole $o$ with radius $r>0$. 

Similarly to Proposition~\ref{prop_ricc}, Theorems~\ref{equiv_G2} and \ref{characterization_biparb_capacity} give the following characterization for the $L^p$-parabolicity of model manifolds.

\begin{proposition}
\label{TSS}Let $M_\sigma$ be a model manifold as above and $p\in \left( 1,\infty
\right) $. Then $M_\sigma$ is $L^{p}$-parabolic if and only if
\begin{equation}
\int^{\infty }\left( \int_{r}^{\infty }\frac{dt}{S(t)}\right) ^{\frac{p}{p-1}%
}S(r)dr=\infty .  \label{SS}
\end{equation}
\end{proposition}

To show that Conjecture~\ref{conjecture} holds for the class of model manifolds we will need the next lemma.


\begin{lemma}
\label{Lemf}Let $f$ be a continuously differentiable function on an interval
$\left( a,b\right) $ such that $f>0$ and $f^{\prime }>0.$ Then%
\begin{equation}
\int_{a}^{b}\frac{dt}{f^{\prime }(t)}\geq \frac{1}{2}\int_{a}^{b}\frac{%
\left( t-a\right) dt}{f(t)}.  \label{fab}
\end{equation}
\end{lemma}

\begin{proof}
Changing $t$ to $t-a$, we reduce \eqref{fab} to
\begin{equation*}
\int_{0}^{b-a}\frac{dt}{f^{\prime }(t+a)}\geq \frac{1}{2}\int_{0}^{b-a}\frac{%
tdt}{f(t+a)}.
\end{equation*}%
Hence, renaming $b-a$ into $b$ and $f(t+a)$ into $f$, we see that it
suffices to prove that
\begin{equation}
\int_{0}^{b}\frac{dt}{f^{\prime }(t)}\geq \frac{1}{2}\int_{0}^{b}\frac{tdt}{%
f(t)}  \label{fab0}
\end{equation}%
that is, \eqref{fab} in the case $a=0.$


Multiplying and dividing by ${f'(t)}^{1/2}$ and using the Cauchy-Schwarz inequality we get
\[
\int_0^b \frac{t}{f(t)} dt \leq
\left(
\int_0^b \frac{t^2}{f(t)^2} f'(t) dt
\right)^{1/2}
\left(
\int_0^b \frac{1}{f'(t)} dt
\right)
^{1/2}.
\]
Integrating by parts in the first integral on the right hand side gives
\[
\int_0^b \frac{t^2}{f(t)^2} f'(t) dt = - \frac{b^2}{f(b)} +2\int_0^b \frac{t}{f(t)}dt\leq 2\int_0^b \frac{t}{f(t)}dt,
\]
whence inserting into the above inequality and simplifying yields \eqref{fab}.
\end{proof}

\begin{corollary}\label{Cor.Vp}
Let $p \in (1,\infty)$. If
\begin{equation}\label{Vp1}
\int^{\infty }r\left( \int_{r}^{\infty }\frac{tdt}{V(t)}\right) ^{\frac{1}{
p-1}}dr=\infty
\end{equation}
then $M_\sigma$ is $L^{p}$-parabolic.
\end{corollary}
\begin{proof}
From a simple change of variables in the external integral we can rewrite \eqref{Vp1} as
\begin{equation*}
\int^{\infty }r\left( \int_{2r}^{\infty }\frac{tdt}{V(t)}\right) ^{\frac{1}{%
p-1}}dr=\infty  .
\end{equation*}%
By Lemma \ref{Lemf}, we have%
\begin{equation*}
\int_{2r}^{\infty }\frac{tdt}{V(t)}\leq 2\int_{r}^{\infty }\frac{\left(
t-r\right) dt}{V(t)}\leq 4\int_{r}^{\infty }\frac{dt}{S(t)}.
\end{equation*}%
Hence, \eqref{Vp1} implies that%
\begin{equation}
\int^{\infty }r\left( \int_{r}^{\infty }\frac{dt}{S(t)}\right) ^{\frac{1}{p-1%
}}dr=\infty .  \label{Sr}
\end{equation}%
Let us show that \eqref{Sr} implies \eqref{SS}. For that, consider the
function
\begin{equation*}
f(r)=\left( \int_{r}^{\infty }\frac{dt}{S(t)}\right) ^{-\frac{1}{p-1}}.
\end{equation*}%
We may assume that $M_\sigma$ is non-parabolic, and therefore $f(r)>0$, for otherwise $M_\sigma$ is automatically $L^p$-parabolic. It follows from \eqref{Sr} that
\begin{equation*}
\int_{1}^{\infty }\frac{rdr}{f(r)}=\infty ,
\end{equation*}%
whence also
\begin{equation*}
\int_{1}^{\infty }\frac{\left( r-1\right) dr}{f(r)}=\infty .
\end{equation*}%
By Lemma \ref{Lemf} we have%
\begin{equation*}
\int_{1}^{\infty }\frac{\left( r-1\right) dr}{f(r)}\leq 2\int_{1}^{\infty }%
\frac{dr}{f^{\prime }(r)},
\end{equation*}%
and, hence,%
\begin{equation}
\int_{1}^{\infty }\frac{dr}{f^{\prime }(r)}=\infty .  \label{f'}
\end{equation}%
We clearly have%
\begin{equation*}
f^{\prime }(r)=\frac{1}{p-1}\left( \int_{r}^{\infty }\frac{dt}{S(t)}\right)
^{-\frac{p}{p-1}}\frac{1}{S(r)}.
\end{equation*}%
Substituting this into \eqref{f'} we obtain \eqref{SS}. Hence, $M_\sigma$ is $L^{p}
$-parabolic by Proposition~\ref{TSS}.
\end{proof}

The next example shows the almost optimality of the pointwise volume condition described in Theorem~\ref{vol_cond_general} and the optimality of the integral volume condition given in Corollaries~\ref{Vp_ricc} and \ref{Cor.Vp}. 

\begin{example}\label{example_optimality}\rm
Let $M_\sigma$ be a model manifold with
$$S(r) = c_m \sigma(r)^{m-1} = c_m r^{\alpha-1} (\log r)^\beta$$
 for $r\geq 2$, with $\alpha> 2$ and real $\beta$.
 Note that in this case
$$V(r) \asymp r^{\alpha} (\log r)^\beta.$$
 Since  $S(r)^{-1}$ is integrable at infinity, the manifold $M_\sigma$ is non-parabolic and its Green kernel with pole at $o$ satisfies for all \(r\gg 1\)
\[
g(r)=\int_{r}^\infty \frac{dt}{S(t)}\asymp r^{2-\alpha }\left( \log r\right) ^{-\beta }.
\]
It follows that
\[
\begin{split}
\left( \int_{r}^{\infty }\frac{dt}{S(t)}\right) ^{\frac{p}{p-1}}S(r) &\asymp r^{\frac{p+1-\alpha }{p-1}}\left( \log r\right) ^{-\frac{\beta }{p-1}}.
\end{split}
\]
Therefore, the integral in (\ref{SS}) is divergent and, hence, $M_\sigma$ is $L^p$-parabolic by Proposition~\ref{TSS},  if and only if either $\alpha<2p$, or $\alpha = 2p$ and $\beta\leq p-1$.
\end{example}

\appendix
\section{}
\label{appendix}


In this appendix, we first establish a regularity property of bounded potentials, which is a required step in the proof of Theorem \ref{parabolicity_measure_dual}. Following this, we present the proof of Theorem \ref{BMP-Thm}, which follows from an adaptation of the original argument in \cite[Thm.E.1]{brezis_marcus_ponce}. The key point in both arguments is the use of convolutions to mollify test functions. In the Riemannian setting, we employ two different approximation procedures: one based on heat semigroup regularization, and the \(L^1\)-approximation introduced in \cite{pigola_valtorta_veronelli_NLA}.

Let $P_t^U$ be the Dirichlet heat semigroup on a bounded smooth open set $U$. Thus, for every $u \in L^2(U)$, the function 
\begin{equation}\label{def_h_t}
u_t(x) = P_t^U u(x) = \int_U p^U_t(x,y)u(y) d\mu(y)
\end{equation}
is a solution to 
\[
\begin{cases}
\partial_t u_t = \Delta u_t  &\text{in } U\times (0,\infty),\\
u_t = 0 &\text{on } \partial U,
\end{cases}
\]
where $p_t^U(x,y)$ is the Dirichlet heat kernel of $U$.
Moreover, if $u \in L^2(U)$, respectively $u\in W^{1,2}_0(U)$ ($=W^{1,2}(U)$ for regular domains)  then, by \cite[Ex.7.37 \& Thm.7.20]{grigoryan_book}
\begin{equation}\label{reg_h_t}
u_t \to u 
\quad 
\text{in} \ L^2(U), \quad \text{ resp. in } W^{1,2}(U), \ \ \text{as} \ t \to 0.
\end{equation}
We note that, for $u \geq 0$ we have $u_t > 0$ in $U$ and
$$ \int_U u_t d\mu \leq \int_U u d\mu.$$
Indeed for every $t>0$,
\begin{eqnarray*}
\frac{d}{dt}\int_U  u_t d\mu &=& \int_U \frac{d}{dt}u_t d\mu \\[0.2cm]
&=& \int_U \Delta u_t d\mu \\[0.2cm]
&=& \int_{\partial U} \langle \nabla u_t, n\rangle d\sigma \leq 0,
\end{eqnarray*}
where $n$ is the outward unit normal vector field along $\partial U$. The last inequality follows from the fact that $u_t>0$ in $U$ and $u_t =0$ on $\partial U$.

We claim that if $u\in W^{1,2}(U)$ then  $P_t^U$ is energy decreasing, that is,
\begin{equation}\label{ineq_nablau_t}
    \int_U \vert \nabla u_t\vert^2 d\mu \leq \int_U \vert \nabla u\vert^2 d\mu.
\end{equation}

\noindent Indeed, differentiation with respect to $t$ gives
\begin{eqnarray*}
\frac{d}{dt}\int_U \langle \nabla u_t, \nabla u_t \rangle d\mu &=& 2\int_U \langle \nabla \frac{d}{dt}u_t, \nabla u_t\rangle d\mu\\[0.2cm]
&=& - 2 \int_U \frac{d}{dt}u_t \Delta u_t d\mu + 2\int_{\partial U} \frac{d}{dt}u_t \langle\nabla u_t,n\rangle d\sigma\\[0.2cm]
&=& - 2\int_U \vert \Delta u_t\vert^2 d\mu\leq 0,
\end{eqnarray*}
where we have used that $u_t = 0$ on $\partial U$ implies $\frac{d}{dt}u_t = 0$ on $\partial U$. 
The conclusion now follows  since $||\nabla u_t||_2\to ||\nabla u ||_2$, as $t\to 0$.

The heat semigroup regularization is a useful tool to prove the following regularity result for bounded potentials.

\begin{lemma}
\label{lemma:hW12}
Let $U$ be a bounded open set with smooth boundary $\partial U$ in a Riemannian manifold $(M,g)$. Given a finite Radon measure $\nu$ supported in  a compact subset $K\subset U$, let $h=G^U \nu$ be the associated potential, so that $\Delta h= -\nu$ in the sense of distributions. If $h\in L^\infty(U)$ then $h\in W^{1,2}(U)$.
%
\end{lemma}
\begin{proof}
Let $p^U_t(x,y)$ denote the Dirichlet heat kernel of $U$ and set 
\[
h_t(x) \doteq P^U_t h(x) =\int_U p_t^U(x,y) h(y) d\mu(y).
\]
%
%
%
Since $-\Delta h=\nu$ in the sense of distributions, $P^U_t$ is symmetric and $P^U_t \Delta \psi=\Delta P_t^U \psi$ for every $\psi \in C^\infty_c(U)$, 
we have
\[
\begin{split}
 \langle \psi, \Delta h_t\rangle &=  \langle   \psi, \Delta P_t^U h\rangle \\
 &=  \langle   \Delta P_t^U  \psi, h\rangle  \\
 &= \langle P_t^U\psi, - \nu \rangle\\
 & =\langle \psi, - P_t^U \nu\rangle = \langle \psi, -\nu_t\rangle,
\end{split}
\]
with $\nu_t=P_t^U \nu$. 
Since both $h_t$ and $\nu_t$ are smooth for $t>0$,
it follows that 
\begin{equation}\label{eq:pde_regularized}
-\Delta h_t=\nu_t .
\end{equation}

The sub-Markovian property of the heat kernel shows that
the semigroup $P^U_t$ is contractive on every $L^p$, 
and therefore 
\begin{equation}\label{eq:linfty_bound}
\|h_t\|_{L^\infty(U)}
\le
\|h\|_{L^\infty(U)}.
\end{equation}
Similarly, 
integrating over $U$ and using Fubini's theorem yield
\[
\begin{aligned}
0\leq \int_U \nu_t(x)\,d\mu(x)
&\le
\int_U
\left(
\int_U p^U_t(x,y)\,d\mu(x)
\right)
d\nu(y) \\
&\le
\int_U d\nu(y)
=
\|\nu\|_{\mathcal M} .
\end{aligned}
\]
Hence
\begin{equation}\label{eq:measure_bound}
\|\nu_t\|_{L^1(U)}
\le
\|\nu\|_{\mathcal M} .
\end{equation}
Multiplying the equality $-\Delta h_t = \nu_t$ by $ h_t$, integrating  over $U$ and using $h_t{_{|\partial U}}=0$, we obtain
\[
\int_U
|\nabla h_t|^2
\,d\mu 
=
\int_U
h_t \nu_t
\,d\mu \leq ||h_t||_{L^\infty(U)} ||\nu_t||_{L^1(U)}\leq  ||h||_{L^\infty(U)} \|\nu\|_{\mathcal M} .
\]
The right-hand side is independent of $t$. Therefore
$\{h_t\}_{t>0}$ is uniformly bounded in $W^{1,2}(U)$.

Since $W^{1,2}(U)$ is reflexive,
bounded sequences have weakly convergent subsequences. Hence, if $\{t_k\}$ is a sequence such that  $t_k\downarrow 0$ then there exist a subsequence   $t_{k'}\downarrow 0$ and $v\in W^{1,2}(U)$ such that $h_{t_{k'}}\to v$ weakly in $W^{1,2}(U)$. In particular, for every $\rho \in C^\infty_c(U)$,
\[
\langle \rho, P^U_{t_{k'}} h\rangle \to \langle \rho, v\rangle = \int_U \rho v \, d\mu, \text{ as } k'\to \infty.
\]
On the other hand, since $ P^U_{t_{k'}} \rho \to \rho$ uniformly, 
\[
\langle \rho, P^U_{t_{k'}} h\rangle 
= \int_U \rho P^U_{t_{k'}} h d\mu = \int_U h  P^U_{t_{k'}} \rho\, d\mu \to
\int_U \rho h\, d\mu, \text{ as } k'\to \infty
\]
and  $h=v\in W^{1,2}(U)$, as required. 
\end{proof}

\begin{remark} 
{\rm By general theory, a solution of $\Delta h = -\nu$ only satisfies $h\in W^{1,p}_{loc}(U)$ for $p\leq n/(n-1)$ (see \cite{Boccardo_Gallouet}). 
The increased regularity depends on the specific form of the function $h$ together with the assumption that $h$ is bounded. We also note that the above proof applies if $\nu$ is a complex compactly supported Radon measure.}
\end{remark}

We now turn to the $L^1$-approximation introduced in \cite{pigola_valtorta_veronelli_NLA}. To this end, let $\tilde U$ be a pre-compact open set with smooth boundary such that $U' \Subset \tilde U \Subset U$ and let $u_r$ be the function defined on $\tilde U$ by
\begin{equation}\label{def_u_r}
u_r(x) = \int_{\tilde U} K_r(x,y) u(y) d\mu(y),
\end{equation}
where $K_r(x,y)$ is the non-symmetric kernel defined by 
\[
K_r(x,y)\doteq \dot \psi \left(g^{\tilde U}(x,y)-r\right) |\nabla_y g^{\tilde U}(x,y)|^2,
\]
and
$\psi:\mathbb R \to [0,1]$ is a smooth function such that
\begin{align*}
\dot \psi \geq 0 \ \, \qquad \qquad \supp\left(\dot \psi\right)\subseteq [-1,1]\ \, , \qquad \qquad \int_{\mathbb R} \dot \psi=1.
\end{align*}
By \cite[Theorem 4.7]{aubin}, there exists a constant $C>0$ such that $g^{\tilde U}(x,y)\leq C d(x,y)^{2-n}$. It follows that for every $R>0$ there exists $r_0=r_0(R)$ such that $K_r(x,y)=0$  for every $x,y$ with $d(x,y)\geq R$ and every $r\geq r_0$. 
Thus, for every $\varepsilon >0 $, $u_r$ is supported in an $\varepsilon$-neighborhood of $\supp u$ for every sufficiently large $r$. In particular, for every $r$ large enough, $u_r$ is compactly supported  in $\tilde U$ and can be extended to a $C^\infty_c$ function on $U$, which we will still denote by $u_r$. 
 
We also remark that, by \cite[Eq. (2.15)]{pigola_valtorta_veronelli_NLA}, 
$$\int_{\tilde U} K_r(x,y) d\mu(y) = 1.$$
Moreover, since  $u$ is superharmonic in $U'$, it follows from \cite[Thm.2.1]{pigola_valtorta_veronelli_NLA} that $u_r$
satisfies
\begin{itemize}
    \item[i)] $u_r \nearrow u$ uniformly on a neighborhood of $K$;
    \item[ii)] the Laplacian of $u_r$ is given by
    \begin{equation}\label{def_Deltau_r}
    \begin{split} \Delta u_r(x) &= \int_{\tilde U} K_r(y,x) \Delta u(y) d\mu(y),
    \end{split}
	\end{equation}
where $\Delta u$ 
is the distributional Laplacian of $u$. 
\end{itemize}
Applying the Fubini-Tonelli theorem, we finally obtain 
\begin{equation}\label{ineq_Lap_u_r}
\int_U \vert \Delta u_r (x)\vert d\mu(x) \leq \Vert \Delta u\Vert_{\mathcal{M}}.
\end{equation}

We are now in a position to prove Theorem~\ref{BMP-Thm}.

\begin{theorem}
Let  $U\subset M$ be  open. Then, for every compact domain $K\subset U$ we have
$$\text{\rm Cap}_{2,1}(K,U) = 2\text{\rm Cap}(K,U).$$
\end{theorem}
\begin{proof}
We first prove that  for every $\varepsilon>0$, there exists $\psi \in C^\infty_{c,+}(U)$ such that $\psi \geq 1$ in some neighborhood of $K$ and
\begin{equation}\label{lemm.E.1}
\int_U \vert \Delta \psi\vert d\mu \leq 2\text{\rm Cap}(K,U) + \varepsilon
\end{equation}
(see \cite[Lemma.E.1]{brezis_marcus_ponce}).
Without loss of generality, we can assume that $U$ is a bounded open set with smooth boundary and that $K$ is a pre-compact subset of  $U$. Let $u$ be the capacitary potential of the condenser $(K,U)$ with respect to the standard capacity, that is, $u \in H^1_0(U)$, $0\leq u \leq 1$, $u=1$ in $K$ and
$$\int_U \vert\nabla u\vert^2 d\mu = \text{\rm Cap}(K,U).$$
Arguing as in the proof of Theorem~\ref{Cap2,2_C_g2}, a standard application of the divergence theorem shows that  $u$ is weakly superharmonic and that $\Delta
u $ is the negative measure supported on $\partial K$  defined by
\[
\langle f,\Delta u \rangle = 
\int _{\partial K} f \langle \nabla u,n\rangle d\sigma,\quad \forall f\in C(U).
\]
where $n$ is the outward unit normal to $K$. 
Since $u=1$ on $\partial K,$ and $u=0$ on $\partial U$, it follows that 
\[\begin{split}
\|\Delta u\|_{\mathcal{M}}&= 
-\int_{\partial K} 1 \langle \nabla u,n\rangle d\sigma =
-\int_{\partial K} u \langle \nabla u,n\rangle d\sigma \\
&= \int_{U\setminus K} |\nabla u|^2 d\mu =  \text{\rm Cap}(K,U).
\end{split}
\]
As in the proof of \cite[Lemma.E.1]{brezis_marcus_ponce}, for $\delta>0$ small we define
$$v = \frac{(u-\delta)_+}{1-2\delta} \in H^1_c(U).$$
Since $\Delta u=0 $ in $U\setminus \partial K$, arguing as above shows that $\Delta v$ is the finite measure supported on  $\partial \{0<v< (1-\delta)(1-2\delta)^{-1}\}=\partial\{\delta <u<1\}$ defined by 
\[
\langle f,\Delta v \rangle = -
\frac 1 {1-2\delta} \int _{\partial\{\delta<u<1\}} f \langle \nabla u,\tilde n\rangle d\sigma,\quad \forall f\in C(U),
\]
where $\tilde n$ denotes the outward pointing unit normal to $\{\delta <u<1\}$.
Furthermore, since $v$ is compactly supported in $U$ and since $v=(u-\delta)/(1-2\delta)$ in a neighborhood of $K,$
\[
0= \langle 1, \Delta v\rangle =
\frac 1 {1-2\delta} \int_{\partial K} \Delta u +\int_{ \partial \{v>0\}} \Delta v,
\]
and therefore
\[
\|\Delta v\|_{\mathcal M} 
= 2\int_{\partial K} \vert \Delta v\vert d\mu = \frac{2}{1-2\delta}\Vert \Delta u\Vert_{\mathcal{M}} = \frac{2}{1-2\delta}\text{\rm Cap}(K,U).
\]
At this point, we consider the $L^1$-approximation $v_r$ of $v$ defined in \eqref{def_u_r}. 
By construction, $v_r \in C^\infty_c(U)$ and $v_r \geq 1$ in some neighborhood of $K$. Moreover, by \eqref{ineq_Lap_u_r}
$$\int_U \vert \Delta v_r (x)\vert d\mu(x) \leq \frac{2}{1-2\delta}\text{\rm Cap}(K,U).$$
Choosing $\delta$ such that $\frac{4\delta}{1-2\delta}\text{\rm Cap}(K,U) < \varepsilon$ we obtain \eqref{lemm.E.1}. Letting $\varepsilon \to 0$ we conclude that 
$$ \text{\rm Cap}_{2,1}(K,U) \leq 2\text{\rm Cap}(K,U).$$

For the other inequality, namely
$$ \text{\rm Cap}(K,U) \leq \frac{1}{2}\text{\rm Cap}_{2,1}(K,U),$$
we consider $\varphi \in C^\infty_c(U)$ such that $\varphi\geq 1$ in some neighborhood of $K$. Set $\tilde \varphi = \min\{1,\varphi_+\} \in C_c(U)\cap H^1_0(U)$ with $0\leq \tilde \varphi \leq 1$. We now take the heat semigroup regularization $\tilde \varphi_t$ of $\tilde \varphi$ defined in \eqref{def_h_t}. 
For $t>0$ sufficiently small, and $\delta>0$ to be chosen, we take $\psi_t = \frac{\tilde \varphi_t}{\sqrt{1-\delta}}$ and compute 
\begin{eqnarray*}
\text{\rm Cap}(K,U) \leq \int_U \vert \nabla \psi_t\vert^2 d\mu = \frac{1}{1-\delta}\int_U \vert \nabla \tilde \varphi_t\vert^2 d\mu \leq \frac{1}{1-\delta}\int_U \vert \nabla \tilde \varphi\vert^2 d\mu,
\end{eqnarray*}
where we used \eqref{ineq_nablau_t} and the fact that $\psi_t \geq 1$ in a neighborhood of $K$ is an admissible test function for the standard capacity. Noting that $\nabla \tilde \varphi = \nabla \varphi$ where $\tilde \varphi$ is not constant, this in turn implies that
\begin{eqnarray*}
\text{\rm Cap}(K,U) \leq -\frac{1}{1-\delta} \int_U  \tilde \varphi \Delta \varphi d\mu.
\end{eqnarray*}
Finally, since $\varphi \in C^\infty_c(U)$, we have that $\int_U \Delta \varphi d\mu = 0$ and 
$$\text{\rm Cap}(K,U) \leq -\frac{1}{1-\delta} \int_U  (\tilde \varphi - \frac{1}{2}) \Delta \varphi d\mu \leq \frac{1}{2(1-\delta)} \int_U \vert \Delta \varphi\vert  d\mu.$$
The required conclusion is obtained by first letting 
$\delta \to 0$ and then taking the infimum over all $\varphi\in C^\infty_c(U)$ such that $\varphi\geq 1$ on $K$. 
\end{proof}

\noindent \textbf{Acknowledgements.} The first author is funded by the Deutsche Forschungsgemeinschaft (DFG, German Research Foundation) - Project-ID 317210226 - SFB 1283. The second author is partially supported by the Alexander von Humboldt Foundation, by Coordena\c c\~ao de Aperfeiçoamento de Pessoal de N\'ivel Superior - Capes (Finance Code 001) and Conselho Nacional de Desenvolvimento Cient\'ifico e Tecnol\'ogico - CNPq, Grants 422900/2021-4, 306543/2022-2. The third author is member of the GNAMPA group “Equazioni differenziali e sistemi dinamici”. The second author is grateful to the Faculty of Mathematics at the Universit\"at Bielefeld for their warm hospitality.

\bibliographystyle{amsplain}

\end{document}